\documentclass[10pt]{amsart}
\usepackage{graphicx,tikz,amssymb,amsmath,amsthm,amsfonts,mathtools,mathrsfs,hyperref, enumerate, enumitem, adjustbox, wasysym, ytableau, youngtab, tikz-cd, faktor}
\usepackage{xcolor} 

\numberwithin{equation}{section}

\makeatletter

\DeclareRobustCommand{\PolishL}{\text{\fontencoding{T1}\selectfont\symbol{138}}}
\makeatother

\newcommand{\symproduct}{\heartsuit}
\newcommand{\symcoproduct}{\spadesuit}
\newcommand{\extproduct}{\diamondsuit}
\newcommand{\extcoproduct}{\clubsuit}

\newcommand{\flip}{\mathsf{flip}}
\def\im{\operatorname{im}}

\def\Sym{\operatorname{Sym}}
\def\SSYT{\operatorname{SSYT}}

\newcommand{\bwedge}[1]{{\textstyle\bigwedge\!\!^{#1}\, }}

\usetikzlibrary{positioning,patterns}
\usepackage{amssymb,amsmath,amsthm,amsfonts,mathrsfs, faktor}

\def\Z{\mathbb Z}
\def\N{\mathbb N} 
 
\def\C{\mathbb C}

\def\F{\mathbb F}

\def\sl{\mathfrak{sl}}
\def\SL{\mathrm{SL}}
\def\x{\mathbf x}
\def\y{\mathbf y}
\def\z{\mathbf z}

\def\sgn{\mathrm {sgn}}

\def\hmi{\what{\mathcal{M}}_i}
\def\mi{\mathcal{M}_i}

\newcommand{\what}[1]{\smash{\widehat{#1}}}

\def\U{\mathcal{U}}

\def\Sym{\operatorname{Sym}}

\renewcommand{\leq}{\leqslant}
\renewcommand{\le}{\leqslant}
\renewcommand{\geq}{\geqslant}
\renewcommand{\ge}{\geqslant}
\renewcommand{\preceq}{\preccurlyeq}

\newcommand{\E}{\!E}

\newtheorem{Proposition}{Proposition}[section]
\newtheorem{Lemma}[Proposition]{Lemma}

\newtheorem{Theorem}[Proposition]{Theorem} 
 
\newtheorem{proposition}[Proposition]{Proposition}
\newtheorem{lemma}[Proposition]{Lemma}
\newtheorem{corollary}[Proposition]{Corollary}

\newtheorem{remark}[Proposition]{Remark}

\theoremstyle{definition}

\newtheorem{Definition}[Proposition]{Definition}

\newcommand{\qbinom}[2]{\genfrac{[}{]}{0pt}{}{#1}{#2}}
\newcommand{\Qbinom}[2]{\genfrac{[}{]}{0pt}{}{#1}{#2}_q}
\newcommand{\Qinteger}[1]{[#1]_q}

\renewcommand{\wr}{\mathrm{wr}}
\newcommand{\SymG}{\mathfrak{S}}

\title[Filtration of plethystic modules for $\SL_2(\F)$]{A field-independent filtration of plethystic modules for $\SL_2(\F)$ that categorifies a product rule for the Cartan subalgebra of $U_q(\sl_2)$}
\author[\'A.~Gutiérrez, \'A.~L.~Mart\'inez, M.~Szwej, M.~Wildon]
{\'Alvaro Guti\'errez, \'Alvaro L. Mart\'inez, \\ Micha\PolishL\ Szwej, and Mark Wildon}
\address{University of Bristol}
\email[Á.~Gutiérrez]{a.gutierrezcaceres@bristol.ac.uk}
\email[M.~Szwej]{michal.szwej@bristol.ac.uk}
\email[M.~Wildon]{mark.wildon@bristol.ac.uk}
\thanks{ÁG was funded by a University of Bristol Research Training Support Grant.
	MS and MW thank the Heilbronn Institute for Mathematical Research for financial support.
}

\begin{document}
\begin{abstract}
We lift a product rule in the Cartan subalgebra of quantum $\mathfrak{sl}_2$ to a filtration of the plethystic representation $\Delta^{(n,m)}\mathrm{Sym}^d E$ of the affine group scheme of the algebraic group $\mathrm{SL}_2$, where $E$ is the natural representation and $\Delta^{(n,m)}$ the Weyl functor. This is a significant step towards a categorification of quantum $\mathfrak{sl}_2$. Our filtration is an addition to a growing family of field-independent isomorphisms of $\mathrm{SL}_2$ representations that include Hermite reciprocity and the Wronskian isomorphism. It is the first such field-independent result requiring multiple filtration layers. It is proved by combinatorial techniques using the authors' symmetric functions model for Weyl modules.
\end{abstract}
\maketitle
\thispagestyle{empty}

\section{Introduction}
Let $\F$ be an arbitrary field,
let $E$ be the natural representation of $\SL_2(\F)$, and let $\Delta^\lambda$ be the Weyl functor 
(see \S\ref{sec:image}) whose formal character is the Schur
function~$s_\lambda$. The main result
of this paper is a field-independent filtration of the $\SL_2(\F)$-representations $\Delta^{(n,m)}\Sym^d\E$. 

\begin{Theorem}\label{thm:main}
    Let $m, n,d\in\mathbb{N}$ with $m\le n,d$. The $\SL_2(\F)$-plethysm $\Delta^{(n, m)}\Sym^d \E$ has the filtration
\begin{align*}
\Sym_{n+m}\Sym^{d-m}\E &\otimes
    \Delta^{(n-m)}\Sym^{m}\E\\
    \Sym_{n+m}\Sym^{d-m+1}\E
    &\otimes \Delta^{(n-m+1, 1)}\Sym^{m-1}\E\\
    &{\,}~\vdots\\
    \Sym_{n+m}\Sym^{d-2}\E
    &\otimes \Delta^{(n-2, m-2)}\Sym^2\E\\
    \Sym_{n+m}\Sym^{d-1}\E
    &\otimes\Delta^{(n-1, m-1)}\Sym^1\E\,
\end{align*}
  from top layer to bottom layer.
\end{Theorem}

The primary motivation for this theorem is that it categorifies
a product rule in the Cartan subalgebra of the quantum group $\U_q(\sl_2)$. We explain this now.
 Let $\mathcal{A} = \Z[q,q^{-1}]$. 
Lusztig's integral form $\U_\mathcal{A}(\sl_2)$ for the Cartan subalgebra
of $\U_q(\sl_2)$ is spanned by the \emph{Lusztig elements}
\[
\qbinom{K;c}{n} = \frac{[K;c][K;c-1]\cdots [K;c-n+1]}{\Qinteger{n}\Qinteger{n-1}\cdots\Qinteger{1}},
\]
where $c \in \N_0$, $n \in \N_0$,
$[K;a] = (q^aK - q^{-a}K^{-1})/(q-q^{-1})$ and $\Qinteger{b} = (q^b - q^{-b})/(q-q^{-1})$ is 
the usual quantum integer. By Theorem~6.2 in the authors' paper \cite{GSMW-Pfaff},
the multiplication rule for the Lusztig elements is
\begin{equation}
    \label{Rel2}
\qbinom{K; c}{n}\qbinom{K; b}{m}=\sum_{i\ge 0} \Qbinom{n-c+b}{i-c}\Qbinom{m-b+c}{i-b}\qbinom{K;i}{n+m},
\end{equation}
where $\Qbinom{n-c+b}{i-c}$ and $\Qbinom{m-b+c}{i-b}$ are the quantum binomial coefficients
defined by 
\[ \Qbinom{n}{k} = \frac{\Qinteger{n}\cdots \Qinteger{n-k+1}}{\Qinteger{k}\ldots \Qinteger{1}}.\]
Formally substituting $q^d$ for $K$, specializing by $b = m$ and $c=n$ and using
the basic identity $\Qbinom{n}{k} = \Qbinom{n}{n-k}$ four times we obtain the $q$-binomial identity
\[
\Qbinom{d+n}{d}\Qbinom{d+m}{d}=\sum_{i\ge 0} \Qbinom{n}{n+m-i}\Qbinom{m}{n+m-i}\Qbinom{d+i}{n+m},
\]
which we interpreted as a form of the $q$-Pfaff--Saalsc{\"u}tz identity and 
proved combinatorially in \cite{GSMW-Pfaff}. Taking the difference of two instances
of this identity we obtain
\begin{align*}
    &\Qbinom{d+n}{d}\Qbinom{d+m}{d} - \Qbinom{d+n+1}{d}\Qbinom{d+m-1}{d} \\
    &\quad =
    \sum_{i\ge 0} \left(\Qbinom{n}{n+m-i}\Qbinom{m}{n+m-i} - \Qbinom{n+1}{n+m-i}\Qbinom{m-1}{n+m-i}\right)\Qbinom{d+i}{n+m}\\
    &\quad =
    \sum_{k\ge 0} \left(\Qbinom{n}{k}\Qbinom{m}{k} - \Qbinom{n+1}{k}\Qbinom{m-1}{k}\right)\Qbinom{d+n+m-k}{n+m}\,.
\end{align*}
By the basic corollary of the Jacobi--Trudi identity in \eqref{eq:twoRowPlethysm}, when $n \ge m$, the left-hand side
is equal to the plethystic product $s_{(n,m)} \circ s_d(q,q^{-1})$.
Similarly the  expression in parentheses in the final line
is $s_{(n-k,m-k)}\circ s_k(q,q^{-1})$. Thus the multiplication rule~\eqref{Rel2} implies the identity
\begin{equation}
    \label{eq:main}
s_{(n,m)}\circ s_d(q,q^{-1}) = 
\sum_{k=1}^{m} 
\Qbinom{d+n+m-k}{n+m}
s_{(n-k,m-k)}\circ s_{k}(q,q^{-1}).
\end{equation}
Interpreting $s_{(n,m)}\circ s_d(q,q^{-1})$ as the character of $\Delta^{(n,m)}\Sym^d \E$,
\smash{$\Qbinom{d+n+m-k}{n+m}$} as the character of $\Sym_{n+m} \Sym^{d-k} \E$ and 
$s_{(n-k,m-k)}\circ s_{k}(q,q^{-1})$ as the character of $\Delta^{(n-k,m-k)}  \Sym^k \E$
(as explained in \S\ref{subsec:plethysmCharacters})
we see that Theorem~\ref{thm:main} is a modular lift of~\eqref{eq:main} that
categorifies~\eqref{Rel2}.

\subsubsection*{A novel categorification}
At the turn of the century, tensor $2$-categories were constructed whose Grothendieck ring is isomorphic to the representation ring of $\U_q(\sl_2)$ \cite{BFK, Stroppel, FKS}. This may be regarded as a  first approximation to the problem of categorifying the quantum group itself. In this setting, the weight spaces of a module are replaced by abelian categories, and the Chevalley generators $E$ and~$F$ act as exact functors between them. Subsequently, an idempotented form of the algebra $\U_q(\sl_2)$ was categorified by Lauda \cite{Lauda}, extending the earlier categorification of the positive part $\U_q^+(\sl_2)$ by Crane and Frenkel \cite{CraneFrenkel} at a root of unity. This lead to the categorification of (idempotented forms of) all quantum groups by Rouquier~\cite{Rouquier} and by Khovanov and Lauda~\cite{KhovanovLauda}. In all of these categorifications, each generator $K$ of the Cartan is replaced with a system of mutually orthogonal idempotents (one for each weight space). None of these earlier categorifications can express~\eqref{Rel2}. Indeed, Theorem~\ref{thm:main} is the first categorification of an instance of the multiplication rule (1.1) requiring a filtration of $\SL_2(\F)$-modules for $\F$ an arbitrary field. As such we believe it is a significant step towards a categorification of $\U_q(\sl_2)$.

\subsubsection*{Modular plethystic isomorphisms}
We emphasise that Theorem~\ref{thm:main} holds over an arbitrary field $\F$.
As such, it provides a filtration of rational representations for the affine
group scheme of the algebraic group $\SL_2$.
It is notable as an addition to a small but growing
family of field-independent isomorphisms
of plethysms, and the \emph{first} that requires a filtration with multiple layers.

The earliest results in this area are \emph{Hermite reciprocity}
\begin{equation}\label{eq:hermite}
\Sym^k\Sym_n\E \cong
\Sym_n\Sym^k\E ,
\end{equation}
and the \emph{Wronskian isomorphism} 
\begin{equation}\label{eq:wr}
\Sym_n\Sym^{k}\E \cong  \bwedge{n}\!\Sym^{n+k-1}\E,
\end{equation}
each proved independently in \cite{AFPRW, McDW}. The latter may be regarded as a modular lift of the
well-known fact that the number of $n$-multisubsets of a set of size $k+1$ is $\binom{n+k}{n}$.
Next in \cite{MW}, the authors proved Theorem~\ref{thm:main}
when $m=1$; this is the very special case when the filtration has only one layer.
Later \cite{hooks} lifted the trinomial revision identity $\binom{M+N}{N}\binom{M+N+d}{M+N} = \binom{N+d}{N}\binom{M+N+d}{M}$ to an isomorphism
\begin{equation}\label{eq:GTL}
\Sym_M\Sym^N\hskip-1pt\E\otimes\Sym_{M+N}\Sym^d\E \cong
\Sym_N\Sym^{d}\E\otimes\Sym_M\Sym^{d+N}\hskip-1pt\E
\end{equation}
and Stanley's hook-content formula (see \cite[Theorem 7.21.2]{StanleyEC2}) for hooks to
\begin{equation}\label{eq:hook}
\Sym_{M}\Sym^{N-1}\E\hskip1pt\otimes\hskip1pt\Sym_{M+N}\Sym^{d-N+1}\E \cong \Delta^{(M+1,1^{N-1})}\Sym^{d}\E\,.
\end{equation}
A very interesting new direction was begun in
\cite{IOT}, where the authors show that
\begin{equation}\label{eq:Kronecker}
(\Sym^{\ell m} \F^{2\ell^2})^{\SL_\ell(\bar\F)\times\SL_\ell(\bar\F)} \cong \Sym_\ell\Sym^m\F^2;
\end{equation}
this is an isomorphism between an invariant ring (related to the Kronecker product) and an
$\SL_2(\F)$-plethysm.

We emphasise the 
jump in complexity that Theorem~\ref{thm:main} represents over these earlier results. 
The techniques developed to prove~\eqref{eq:GTL} and \eqref{eq:hook} are extended in the proof of Theorem~\ref{thm:main}; they are very different from the algebraic geometry approach in \cite{AFPRW} or the basis dependent methods used in \cite{MW}. We expect that 
the new and mainly combinatorial techniques presented in~\S\ref{sec:aux} and~\S\ref{sec:proof}
will be very useful in further proofs of characteristic-free isomorphisms for $\SL_2$.

\section{Preliminaries}\label{sec:preliminaries}
\subsection{Partitions}\label{subsec:partitions}
We follow \cite{StanleyEC2}.
A \emph{partition} $\lambda = (\lambda_1, \lambda_2, \ldots, \lambda_k)$ is a weakly decreasing sequence of positive integers. We say $\ell(\lambda) = k$ is its \emph{length} and $|\lambda| = \lambda_1+\cdots+\lambda_k$ its \emph{size}. 
By convention, we let $\lambda_L = 0$ if $L > \ell(\lambda)$.
We write exponents for repeated parts: for instance~$(7,4^2,3,1^3) = (7,4,4,3,1,1,1)$. In this paper, partitions of two parts are often the indices of maps, in which case we sometimes abbreviate $f_{(a,b)}$ by $f_{ab}$.

The \emph{sum} of partitions $\lambda+\mu$ is entry-wise, for instance~$(3,2,1)+(4,3,3,2) = (7,5,4,2)$. The \emph{union} of partitions $\lambda\sqcup\mu$ is the sorted list of the multiset union of their parts, for instance~$(3,2,1)\sqcup(4,3,3,2) = (4,3^3,2^2,1)$.

The \emph{Young diagram} of a partition $\lambda$ is the set 
\[
[\lambda] = \{(i,j)\in\Z^2\mid 1\le i\le \ell(\lambda),\, 1\le j\le \lambda_i\}
\]
which we represent in the English convention as a top-left justified array of boxes,
\[
\ytableausetup{smalltableaux, centertableaux}
[(7,4,4,2)] = \ydiagram{7,4,4,2}.
\]
The \emph{transpose} $\lambda'$ of a partition $\lambda$ is the partition whose Young diagram is $\{(i,j)\mid (j,i)\in[\lambda]\}$.
We write $\mu\subseteq\lambda$ if $[\mu]\subseteq[\lambda]$. In that case, the \emph{skew partition} $\lambda/\mu$ is the formal pair $(\lambda,\mu)$, and the Young diagram of the skew partition is the set $[\lambda/\mu]= [\lambda]\setminus[\mu]$. Given $n, d \in \N_0$ let $L(n,d) = \{\lambda\in\mathrm{Par} \mid 
    \lambda_1\le d,~\ell(\lambda)\le n\}$ be the set of partitions fitting in an $n\times d$ box.

A \emph{tableau} is a map $T : Y \to [n]$, where $Y$ is the Young diagram of a partition~$\lambda$ or a skew partition $\lambda/\mu$. We say $\lambda$ or $\lambda/\mu$ is the \emph{shape} of the tableau and $[n]$ is the \emph{alphabet}.
The \emph{weight} of a tableau $T$ is the tuple $w(T) = (w_1, w_2, \ldots, w_n)$ where $w_i$ counts the number of
boxes in $[\lambda]$ or $[\lambda/\mu]$ sent to $i$ by $T$.
A tableau is \emph{semistandard} if
\[
T(i,j) \le T(i, j+1)
\quad\text{and}\quad
T(i,j) < T(i+1, j)
\]
wherever the function is defined. We represent a tableau by placing $T(i,j)$ inside the box $(i,j)$ of the Young diagram; then a tableau is semistandard if it is weakly increasing when rows are read left-to-right and strictly increasing 
when columns are read top-to-bottom.

\subsection{Alternating and symmetric polynomials}
Our general reference for symmetric polynomials and symmetric functions
is \cite[Ch.~7]{StanleyEC2}.
Let $\x_n = (x_1, \ldots, x_n)$.
Let the symmetric group $\SymG_n$ act on the polynomial ring $\F[\x_n]$ by permuting the variables.
A polynomial $p(\x_n)$ is \emph{alternating} if $\sigma \cdot p(\x_n) = -p(\x_n)$ for all transpositions $\sigma \in \SymG_n$ and $p(\ldots, x, x, \ldots) = 0$ whenever two entries coincide. (Note the second condition is needed
if $\F$ has characteristic two.)
A polynomial is \emph{symmetric} if it is invariant under the $\SymG_n$ action.

Let $\rho_n = (n-1, n-2, \ldots, 1,0)$ be the staircase partition. As $\lambda$ ranges over all partitions of length at most $n$, the determinants 
\[
a_{\lambda + \rho_n}(\x_n) = \det(x_j^{\lambda_i+(\rho_n)_i}) = \det(x_j^{\lambda_i+n-i})
\]
form a basis of the $\F$-vector space $A[\x_n]$ of alternating polynomials. 
When $\lambda = \varnothing$, then $a_{\rho_n}(\x_n) = \prod_{i<j} (x_i-x_j)$ is the Vandermonde determinant.
The $\F$-vector space $\Lambda[\x_n]$ of symmetric polynomials satisfies 
\[
a_{\rho_n}(\x_n) \, \Lambda[\x_n] = A[\x_n].
\]

\subsection{Schur polynomials and Littlewood--Richardson coefficients}
The corresponding basis of $\Lambda[\x_n]$ is that of \emph{Schur polynomials},
\[
s_\lambda(\x_n) = \frac{a_{\lambda + \rho_n}(\x_n)}{a_{\rho_n}(\x_n)}.
\]
A classic result is that 
\[
s_\lambda(\x_n) = \sum_{T\in\SSYT_n(\lambda)} \x_n^{w(T)},
\]
where $\SSYT_n(\lambda)$ is the set of semistandard Young tableaux of shape $\lambda$ in the alphabet $[n]$, and where  $\x_n^{(\alpha_1, \ldots, \alpha_n)} = x_1^{\alpha_1}\cdots x_n^{\alpha_n}$.
More generally, the \emph{skew Schur polynomials} are defined as
\[
s_{\lambda/\mu}(\x_n) = \sum_{T\in\SSYT_n(\lambda/\mu)} \x_n^{w(T)}
\]
for skew partitions $\lambda/\mu$.

The \emph{Littlewood--Richardson coefficients} are the structure constants $c_{\mu\nu}^\lambda$ of the product of Schur polynomials,
\[
s_\mu(\x_n)s_\nu(\x_n) = \sum_{\lambda} c_{\mu\nu}^\lambda s_\lambda(\x_n).
\]
They are also the constants expressing skew Schur polynomials in the Schur basis,
\[
s_{\lambda/\mu}(\x_n) = \sum_{\nu} c_{\mu\nu}^\lambda s_\nu(\x_n).
\]
The \emph{reverse reading word} $\mathrm{word}(T)$ of a tableau $T$ is the sequence of its entries obtained by reading rows right-to-left, beginning with the topmost row and finishing with the bottom-most. A sequence $(t_1, \ldots, t_L)$ is said to be a \emph{ballot sequence} if at every prefix $(t_1, \ldots, t_k)$ there are at least as many instances of $i$ as of $i+1$, for every $i=1,\ldots,n-1$. Defining
\[ \mathrm{LR}_{\mu,\nu}^\lambda(n) = 
\{
T\in\SSYT_n(\lambda/\mu) \mid 
\text{$w(T)=\nu$ and $\mathrm{word}(T)$ is a ballot sequence}
\},
\]
we have $c_{\mu\nu}^\lambda = \# \mathrm{LR}_{\mu,\nu}^\lambda(n)$.
The elements of $\mathrm{LR}_{\mu,\nu}^\lambda(n)$ are called
\emph{Littlewood--Richardson tableaux}.

\subsection{Tensor spaces}\label{sec:tensor spaces}
Given an $\F$-vector space $V$ let
\[
\bwedge{n}V = \faktor{V^{\otimes n}}{( \cdots \otimes w \otimes w \otimes  \cdots )}
\]
be its \emph{exterior power}. Note this definition is correct even when $\F$
has characteristic~$2$, and that setting $w = u+v$ gives
the relation $\cdots u \otimes v \cdots + \cdots v \otimes u \cdots$,
as expected for the exterior power.
There are two symmetric power functors, the \emph{upper} and \emph{lower symmetric powers}
\[
\Sym^nV = \faktor{V^{\otimes n}}{(\cdots u\otimes v\cdots - \cdots v\otimes u\cdots)}
\quad
\text{and}
\quad
\Sym_nV = (V^{\otimes n})^{\SymG_n},
\]
where the symmetric group $\SymG_n$ acts by permuting tensor factors.
Typically $\Sym^nV$ and $\Sym_nV$ are not  isomorphic over $\F$ 
as representations of $\mathrm{GL}(V)$, but they are always isomorphic working over $\C$.
\begin{remark}\label{rem:lower-exterior}
    Analogously, one could consider the \emph{lower exterior power}
    defined as the invariants under the \emph{signed} action of $\mathfrak{S}_n$. Here, in contrast to the symmetric powers, a routine computation shows that it is $\mathrm{GL}(V)$-isomorphic to the exterior power defined above, for any field $\F$. Throughout the paper, we write $\bwedge{n} V$ for both cases, as the construction will be usually clear from the context.
\end{remark}

Throughout this paper we work with many-fold tensor products of representations. We introduce the following shorthand notation for convenience.
\begin{Definition}
    Let $V$ be an $\F$-vector space.
    Given a partition $\lambda$, let
    \[
    V^{\otimes\lambda} = \bigotimes_{(i,j)\in[\lambda]} V(i,j),
    \]
    where $V(i,j)$ is an isomorphic copy of $V$.
    Let 
    \[
    \Sym_\lambda\! V = \Sym_{\lambda_1}\! V\otimes \dots\otimes\Sym_{\lambda_{\ell(\lambda)}}\! V
    \]
    and
    \[
    \bwedge{\lambda}\hskip-1pt V = \bwedge{\lambda_1'}\hskip-1ptV\otimes\dots\otimes\bwedge{\lambda'_{\lambda_1}}\hskip-1ptV\,.
    \]
\end{Definition}

We introduce a pictorial way of representing the three objects above: we represent $V^{\otimes\lambda}$ by the Young diagram of~$\lambda$, the space $\Sym_\lambda V$ by a row tabloid of shape~$\lambda$, and 
$\bwedge{\lambda}V$ by a column tabloid, in each case filled with the coordinates of the boxes.
For instance, for $\lambda=(5,3,3)$:
\[
V^{\otimes\lambda} =
\ytableausetup{centertableaux,boxsize=.395cm}
{\tiny\ytableaushort{{11}{12}{13}{14}{15},{21}{22}{23},{31}{32}{33}}},
\quad
\Sym_\lambda\! V = 
\begin{tikzpicture}[x=.4cm,y=.4cm,baseline=.5cm]
    \draw 
        (0,3) rectangle ++ (5,-1)
        (0,2) rectangle ++ (3,-1)
        (0,1) rectangle ++ (3,-1);
    \foreach\i in {1,...,5}{
    \node (1\i) at (\i-.5,2.5) {\tiny 1\i};
    }
    \foreach\i in {1,...,3}{
    \node (2\i) at (\i-.5,1.5) {\tiny 2\i};
    \node (3\i) at (\i-.5,0.5) {\tiny 3\i};
    }
\end{tikzpicture},
\quad\text{and}\quad
\bwedge{\lambda}\hskip-1pt V = 
\begin{tikzpicture}[x=.4cm,y=.4cm,baseline=.5cm]
    \draw 
        (0,3) rectangle ++ (1,-3)
        (1,3) rectangle ++ (1,-3)
        (2,3) rectangle ++ (1,-3)
        (3,3) rectangle ++ (1,-1)
        (4,3) rectangle ++ (1,-1);
        \foreach\i in {1,...,5}{
    \node (1\i) at (\i-.5,2.5) {\tiny 1\i};
    }
    \foreach\i in {1,...,3}{
    \node (2\i) at (\i-.5,1.5) {\tiny 2\i};
    \node (3\i) at (\i-.5,0.5) {\tiny 3\i};
    }
\end{tikzpicture}.
\]
\subsection{Product and coproduct}\label{subsec:productCoproduct}
Following~\cite[\S1.1.1]{Wey03}, there is a product $\symproduct$ and coproduct $\symcoproduct$ for symmetric powers
\[\begin{tikzcd}
	{\Sym_{\lambda}\!V} && {\Sym_{|\lambda|}\!V}
	\arrow["{\symproduct_{\lambda}}", shift left, from=1-1, to=1-3]
	\arrow["{\symcoproduct_{\lambda}}", shift left, from=1-3, to=1-1]
\end{tikzcd}\!\!\!,\]
as well as a product $\extproduct$ and coproduct $\extcoproduct$ for exterior spaces,
\[\begin{tikzcd}
	{\bwedge{\lambda}\hskip-1ptV} && {\bwedge{|\lambda|}V}
	\arrow["{\extproduct_{\lambda}}", shift left, from=1-1, to=1-3]
	\arrow["{\extcoproduct_{\lambda}}", shift left, from=1-3, to=1-1]
\end{tikzcd}\!\!.\]
We drop the subscripts of the maps when they are understood from context.

\begin{remark}\label{remark:ExtProd}
    Let $\lambda = (\lambda_1,\ldots,\lambda_\ell)$ be a partition of length $\ell$ and size $M$.
    The products are best understood as elements of the group algebra of the symmetric group $\SymG_M$
    which acts on $V^{\otimes M}$ by permuting tensor factors. The action extends linearly to $\F\SymG_M$.
 The product $\symproduct$ can be identified with the group algebra element
    \[
    \symproduct_{\lambda} = \sum_{\sigma\in\mathrm{Shuff}(\lambda)} \sigma
    \]
    and $\extproduct$, in light of Remark~\ref{rem:lower-exterior}, with the group algebra element
    \[
    \extproduct_{\lambda'} = \sum_{\sigma\in\mathrm{Shuff}(\lambda)} \sgn(\sigma) \sigma,
    \]
    where $\mathrm{Shuff}(\lambda)$ is a set of minimal length coset representatives of the quotient by the Young subgroup, $\SymG_{M}/(\SymG_{\lambda_1}\times\cdots\times \SymG_{\lambda_\ell})$. 
    These representatives are called \emph{shuffles}, and in the one-line notation they are the permutations of $M$ such that each of the subsets $\{1,\ldots,\lambda_1\}, \{\lambda_1+1, \ldots, \lambda_1+\lambda_2\}, \ldots$ appears in linear order \cite[Ch.~2, Exercise 4]{BjornerBrenti}.
    As elements of the group algebra, $\symcoproduct$ and (by Remark~\ref{rem:lower-exterior}) $\extcoproduct$ are simply the identity.
\end{remark}

\subsection{Weyl modules as images}\label{sec:image}
For the construction of Weyl modules we follow \cite[p.~43]{Wey03} or \cite{McDowell25};
the latter construction is equivalent but more elementary and avoids divided symmetric powers.
Again let $V$ be an $\F$-vector space.
The coproducts
\[
\symcoproduct_{(1^{\lambda_j})} : 
\Sym_{\lambda_j}\! V\hookrightarrow V(j, 1)\otimes\dots\otimes V(j, \lambda_j)
\]
induce a map
\(
\symcoproduct^\lambda := \symcoproduct_{(1^{\lambda_1})}\otimes\cdots\otimes\,\symcoproduct_{(1^{\lambda_{\ell(\lambda)}})} :
\Sym_\lambda\! V \hookrightarrow V^{\otimes \lambda}.
\) 
Similarly, the products
\begin{align*}
    \extproduct_{(\lambda'_i)} : 
    V(1, i)\otimes\dots\otimes V(\lambda'_i, i)&\twoheadrightarrow \bwedge{\lambda'_i}\hskip-1pt V\\
    v_1\otimes\cdots\otimes v_{\lambda'_i} &\mapsto 
    v_1\wedge\cdots\wedge v_{\lambda'_i}
\end{align*}
induce a map
\(
\extproduct^\lambda = \extproduct_{(\lambda'_1)}\otimes\cdots\otimes\extproduct_{(\lambda'_{\lambda_1})} : V^{\otimes\lambda}\twoheadrightarrow\bwedge{\lambda}\hskip-1pt V.
\)

\begin{Definition}\label{def:weyl-as-image}
    The \emph{Weyl module} $\Delta^{\lambda}V$ is defined to be the image of the composition $\psi_\lambda=\extproduct^\lambda\circ\symcoproduct^\lambda$:
    \begin{equation}
        \label{eq:psi lam}
    \Delta^\lambda V = 
    \im(\extproduct^{\lambda}\circ\symcoproduct^{\lambda} :
    \Sym_\lambda\! V\xrightarrow{\psi_\lambda} \bwedge{\lambda}\hskip-1pt V).
    \end{equation}
\end{Definition}

Pictorially, $\psi_\lambda$ is the composite
\[
\psi_\lambda : 
\begin{tikzpicture}[x=.4cm,y=.4cm,baseline=.5cm]
    \draw 
        (0,3) rectangle ++ (5,-1)
        (0,2) rectangle ++ (3,-1)
        (0,1) rectangle ++ (3,-1);
    \foreach\i in {1,...,5}{
    \node (1\i) at (\i-.5,2.5) {\tiny 1\i};
    }
    \foreach\i in {1,...,3}{
    \node (2\i) at (\i-.5,1.5) {\tiny 2\i};
    \node (3\i) at (\i-.5,0.5) {\tiny 3\i};
    }
\end{tikzpicture}
\hookrightarrow
{\tiny\ytableaushort{{11}{12}{13}{14}{15},{21}{22}{23},{31}{32}{33}}}
\twoheadrightarrow
\begin{tikzpicture}[x=.4cm,y=.4cm,baseline=.5cm]
    \draw 
        (0,3) rectangle ++ (1,-3)
        (1,3) rectangle ++ (1,-3)
        (2,3) rectangle ++ (1,-3)
        (3,3) rectangle ++ (1,-1)
        (4,3) rectangle ++ (1,-1);
        \foreach\i in {1,...,5}{
    \node (1\i) at (\i-.5,2.5) {\tiny 1\i};
    }
    \foreach\i in {1,...,3}{
    \node (2\i) at (\i-.5,1.5) {\tiny 2\i};
    \node (3\i) at (\i-.5,0.5) {\tiny 3\i};
    }
\end{tikzpicture}.
\]
Since the construction is functorial, this defines $\Delta^\lambda V$
as a representation of $\mathrm{GL}(V)$, and if $V$ is a representation of
$\SL_2(\F)$, by restriction we see that $\Delta^\lambda V$ is a representation
of $\SL_2(\F)$.

\subsection{The characters of Weyl modules and plethysm}\label{subsec:plethysmCharacters}
The formal character of the Weyl module $\Delta^\lambda V$
as a representation of $\mathrm{GL}(V)$ is the Schur polynomial
$s_\lambda(\x_{d+1}) = s_\lambda(x_1, \ldots, x_{d+1})$, where $d+1=\dim V$. 
Recall that~$E$ is the natural $2$-dimensional representation of $\SL_2(\F)$.
If $V$ is the representation $\Sym^d \E$ of $\SL_2(\F)$
then the eigenvalues of the diagonal matrix $\left( \begin{smallmatrix} q & 0 \\ 0 & q^{-1} \end{smallmatrix}\right) \in \SL_2(\F)$ 
acting on $V$ are $q^d, q^{d-2}, \ldots, q^{-d}$. Thus,
the \emph{$q$-character} of $\Sym^d \E$ is
\begin{equation} q^d + q^{d-2} + \cdots + q^{-d} = [d+1]_q\,.\label{eq:qCharacterSym} \end{equation}
Moreover, the character of $\Delta^\lambda \Sym^d \E$
on this matrix is obtained by substituting these eigenvalues for the formal variables $x_1, x_2,\ldots$,
and so the $q$-character of $\Delta^\lambda \Sym^d \E$ is
\begin{equation}\label{eq:qPlethysm} 
s_\lambda(q^d, q^{d-2}, \ldots, q^{-d}) = s_\lambda \circ s_d(q,q^{-1}) \end{equation}
where the equality 
follows from $s_d(q,q^{-1}) = q^d + q^{d-2} + \cdots + q^{-d}$
and the definition of the plethysm product by monomial substitution
(see \cite[Theorem 7]{LoehrRemmel}).

\begin{remark}
More generally, composition of Weyl functors $\Delta^\lambda$ corresponds
to the plethysm product on Schur functions: see 
\cite[page 448]{StanleyEC2},
or \cite[160, (7.3)]{Macdonald} for an even more general result on polynomial functors.
The simplified treatment in this subsection suffices for our purposes.
\end{remark}

As a corollary of~\eqref{eq:qPlethysm} we prove
the result relating $q$-binomial coefficients to characters
of two-row Weyl modules used in the derivation of~\eqref{Rel2}.
Let $d, m, n \in \N_0$ with $n \ge m$. It follows from the specialism
$h_n(1,q,\ldots, q^{d})$ of the complete homogeneous symmetric function
of degree $n$ in \cite[Ch.~1, \S 2, Example 3]{Macdonald} that
$\Qbinom{d+n}{d} = h_n(q^d,\ldots, q^{-d})$. Hence
\begin{align} \Qbinom{d+n}{d}\Qbinom{d+m}{d} &- \Qbinom{d+n+1}{d}\Qbinom{d+m-1}{d} \notag \\
&\ = (h_n h_m - h_{n+1}h_{m-1})(q^d, q^{d-2}, \ldots, q^{-d}) \notag \\
&\ = s_{(n,m)}(q^d, q^{d-2}, \ldots, q^{-d}) \notag \\
&\ = s_{(n,m)} \circ s_d(q,q^{-1}) \label{eq:twoRowPlethysm} \end{align}
where we used the Jacobi--Trudi identity and then~\eqref{eq:qPlethysm}.
See \cite[\S 6, (3)]{GutKrat} for other applications of this identity.

\subsection{Weyl modules as vector spaces of polynomials}
\subsubsection*{One-row and one-column Weyl modules}
    In \cite{hooks} we constructed the isomorphism of $\SL_2(\F)$-modules
    \[
    \Sym_n\Sym^d\E \cong \Lambda_{\le d}[\x_n],
    \]
    where the right-hand side is the vector space of symmetric polynomials in $\x_n = (x_1, \ldots, x_n)$ with coefficients in $\F$ and of degree at most $d$ \emph{in each variable separately}. 
    An $\F$-basis of $\Lambda_{\le d}[\x_n]$ is given by
    \[
    \{
    s_\lambda(\x_n) \mid
    \lambda\in L(n,d)
    \},
    \]
    where $s_\lambda$ is a Schur polynomial and, as defined in \S\ref{subsec:partitions}, 
    $L(n,d) = \{\lambda\in\mathrm{Par} \mid 
    \lambda_1\le d,~\ell(\lambda)\le n\}$ is the set of partitions fitting in an $n\times d$ box \cite{GMSW-FPSAC}. The group $\SL_2(\F)$ acts on $\Lambda_{\leq d}[\x_n]$ via
\begin{equation}\label{eq:SL2action}
\left(\begin{matrix}
	\alpha&\beta\\\gamma&\delta
\end{matrix}\right) f(x_1,\ldots,x_n)=\prod_{i=1}^n(\beta x_i+\delta)^d\cdot f\!\left(
    \frac{\alpha x_1+\gamma}{\beta x_1+\delta},
    \ldots,
    \frac{\alpha x_n+\gamma}{\beta x_n+\delta}
    \right).
\end{equation}
To keep track of the names of the variables and their maximum degree, we have the following pictorial representation:
\[
\Sym_n\Sym^d\E =
\Lambda_{\le d}[x_1, \ldots, x_n] = 
        \begin{tikzpicture}[x=.45cm,y=.45cm,baseline=.15cm]
        \draw 
            (0,1) rectangle ++ (4,-1);
            \node(13) at (0.5, .5){\small$x_{1}$};
            \node(13) at (1.5, .5){\small$x_{2}$};
            \node(14) at (2.5, .5){\small$\cdots$};
            \node(14) at (3.5, .5){\small$x_{n}$};
        \end{tikzpicture}_{d}\,.
\]

    \begin{remark}\label{rem:sym-coprod-schur}
        The symmetric
        coproduct $\symcoproduct$ becomes the usual coproduct inherited from the (plethystic) Hopf algebra structure of $\Lambda[\x]$ (see \cite[\S5.1]{ALRS}).
        On the Schur basis,
        \begin{equation}
            \label{eq:Schur coproduct}
        \symcoproduct_{ab}\, s_\lambda(\x_a,\y_b) = \sum_{\mu\subseteq\lambda} s_\mu(\x_a)\otimes s_{\lambda/\mu}(\y_b)\,.
        \end{equation}
        Suppose $\lambda\in L(a+b,k)$. Note that $s_\mu(\x_a)\neq0$ requires $\ell(\mu)\leq a$, so $\mu\in L(a,k)$. Similarly, for $s_{\lambda/\mu}(\y_b)\neq0$, each column of the skew partition $\lambda/\mu$ has to be of height at most $b$, equivalently $\mu\subseteq\lambda\subseteq\mu\sqcup(k^b)$, thus we may rewrite the above as
        \begin{equation}
            \label{eq:Schur coproduct 2}
        \symcoproduct_{ab}\, s_\lambda = \sum_{\substack{\mu\in L(a,k)\\\mu\subseteq\lambda\subseteq\mu\sqcup(k^b)}} s_\mu\otimes s_{\lambda/\mu}\,.
        \end{equation}
    \end{remark}
    
A single-column Weyl module $\bwedge{n}\hskip-1pt\Sym^d\E$ is identified with the $\F$-vector space $A_{\le d}[\x_n]$ of alternating polynomials in $n$ variables, each of degree at most $d$. We have a similar pictorial representation:
\[
\bwedge{n}\Sym^d\E =
A_{\le d}[x_1, \ldots, x_n] = 
        \begin{tikzpicture}[x=.45cm,y=.45cm,baseline=-2em,rotate=-90]
        \draw 
            (0,1) rectangle ++ (3,-1);
            \node(13) at (0.5, .5){\small$x_{1}$};
            \node(14) at (1.5, .5){\raisebox{.5em}{\small$\vdots$}};
            \node(14) at (2.5, .5){\small$x_{n}$};
        \end{tikzpicture}_{d}.
\]
A basis of $\bwedge{n}\hskip-1pt\Sym^d \E$ is therefore $\{a_{\lambda+\rho_n}(\x_n)\mid\lambda\in L(n, d-n+1)\}$.

\begin{remark}\label{rmk:ext-product}
    The exterior product $\extproduct_{(k,\ell)'}$ also admits a convenient description on the basis above. For $\lambda\in L(k,d-k+1)$ and $\mu\in L(\ell,d-\ell+1)$,
    \begin{equation*}
        \extproduct_{(k,\ell)'}(a_{\lambda+\rho_k}(\x_k)\otimes a_{\mu+\rho_\ell}(\y_\ell))=\sgn_{k\ell}(\lambda,\mu)a_{(\lambda+\rho_k)\sqcup(\mu+\rho_\ell)}(\x_k,\y_\ell),
    \end{equation*}
    with
    \begin{equation*}
        \sgn_{k\ell}(\lambda,\mu)
        =
        \begin{cases}
            \sgn(\tau)&\text{if the parts of }(\lambda+\rho_k)\sqcup(\mu+\rho_\ell)\text{ are distinct,}\\
            0&\text{otherwise,}
        \end{cases}
    \end{equation*}
    where $\tau$ is the permutation sorting the concatenated sequence of parts into decreasing order. Explicitly,
    \[
    \sgn(\tau)=(-1)^{\#\mathrm{inv}(\lambda_1+k-1,\,\lambda_2+k-2,\,\ldots,\,\lambda_k,\,\mu_1+\ell-1,\,\mu_2+\ell-2,\,\ldots,\,\mu_\ell)}.
    \]
\end{remark}

\subsubsection*{Wronskian and the (co)products}
The interpretation of Weyl modules as $\F$-vector spaces of polynomials becomes a useful setting to study the field-independent morphisms \eqref{eq:wr}--\eqref{eq:hook}. For instance the Wronskian isomorphism~\eqref{eq:wr} is realised as
\begin{align}\label{eq:wr as poly}
    \wr_n : \Lambda_{\le k} [\x_n] &\to A_{\le n+k-1}[\x_n]\notag\\
    p(\x_n) &\mapsto a_{\rho_n}(\x_n)\,p(\x_n).
\end{align}
See~\cite{McDW} for the original map and \cite{Grinberg, hooks} for the interpretation above.

The following computation in the group algebra is useful later. Write $\x_n^{\alpha}$ for the operator that multiplies a polynomial by $x_1^{\alpha_1}\cdots x_{n}^{\alpha_n}$. Recall from \S\ref{subsec:productCoproduct}
that $\extproduct$ is the product for exterior powers, $\symcoproduct$ is the coproduct for symmetric powers, and $\psi$ is the composition from Definition~\ref{def:weyl-as-image}.
\begin{lemma}\label{lem:commute}
    $\extproduct_{(2^k)} \wr_k^{\otimes2} = \extproduct_{kk}\x_k^{\rho_k}\y_k^{\rho_k}\psi_{kk} $.
\end{lemma}
\begin{proof}
    Begin by noting $a_{\rho_k}(\x_k) = \sum_{\tau\in \SymG_k} \sgn(\tau) \tau \cdot\x_k^{\rho_k}$.
    By Remark~\ref{remark:ExtProd} the left-hand side becomes the group algebra element
\[
    \sum_{\sigma\in\mathrm{Shuff}(k,k)}
    \sum_{\tau_1\in \SymG_k} \sum_{\tau_2\in \SymG_k}
    \sgn(\sigma)\sgn(\tau_1)\sgn(\tau_2) ~\sigma(\tau_1\cdot \x_k^{\rho_k}\otimes\tau_2\cdot \y_k^{\rho_k})
\]
    where we have used two alphabets $\x$ and $\y$ to distinguish tensor factors. Now using 
    the fact that $\mathrm{Shuff}(k,k)$ is a set of minimal length coset representatives of $\SymG_{2k}/(\SymG_{k}\times\SymG_{k})$ this becomes $\sum_{\pi\in\SymG_{2k}}
    \sgn(\pi) ~\pi\cdot\x_k^{\rho_k}\y_k^{\rho_k}$.

    To write the right-hand side we follow the action of the different symmetric groups. The operator is
    \begin{multline*}
        \extproduct_{kk} \x_k^{\rho_k}\y_k^{\rho_k} (\extproduct_{2} \,\otimes\, 
        \stackrel{k}{\cdots}\,\otimes\,\extproduct_{2})(\symcoproduct_{(1^k)}\otimes\symcoproduct_{(1^k)})\\
        = 
        \sum_{\sigma\in\mathrm{Shuff}(2^k)}
        \sum_{\substack{\tau_1\in\SymG_2\\\cdots\\\tau_k\in\SymG_2}} \sgn(\sigma)\sgn(\tau_1)\cdots\sgn(\tau_k)\sigma\cdot\x_k^{\rho_k}\y_k^{\rho_k}~ \tau_1\cdots\tau_k,
    \end{multline*}
    where the different copies of the symmetric group
    $\SymG_2$ act on the sets $\{x_i,y_i\}$. Note that $\tau_i(\x_k^{\rho_k}\y_k^{\rho_k})=\x_k^{\rho_k}\y_k^{\rho_k}$, so the $\tau_i$'s commute with $\x_k^{\rho_k}\y_k^{\rho_k}$. The expression becomes
    \[
       \sum_{\pi\in\SymG_{2k}}
        \sgn(\pi)\,\pi\cdot\x_k^{\rho_k}\y_k^{\rho_k}. \qedhere
    \]
\end{proof}

\subsubsection*{The evaluation map}
The $\SL_2(\F)$-isomorphisms in \cite{hooks} were stated and proved using  the evaluation map 
studied here.
A generalisation and restriction of this map play an important role in the proof of Theorem~\ref{thm:main}.
We begin by introducing the most general version of the evaluation map required.
\begin{Definition}\label{def:pi-map}
    Let $\lambda$ be a partition and let $N,a,b\in\N_0$. The \emph{evaluation map} is defined by
    \begin{align}\label{eq:pi}
        \pi_{\lambda,N,a,b}:\Sym_{|\lambda|+N}\Sym^b \E\otimes (\Sym^a\E)^{\otimes\lambda}
        &\to \Sym_{N}\Sym^b\E\otimes (\Sym^{a+b}\E)^{\otimes\lambda}\notag\\
        P(\x_{|\lambda|},\y_{N})Q(\z_{|\lambda|})&\mapsto P(\z_{|\lambda|},\y_N)Q(\z_{|\lambda|}).
    \end{align}
    The parameters $\lambda,N,a,b$ are usually clear from the context, in which case we write just~$\pi$ instead of $\pi_{\lambda, N, a, b}$.
\end{Definition}

\begin{Proposition}\label{p:injectivity pi}
    The map $\pi_{\lambda, N,a,b}$ is well defined and $\SL_2(\F)$-equi\-va\-riant. If $N \ge a$ then $\pi_{\lambda,N,a,b}$ is injective for any partition $\lambda$ and $b\in \N_0$.
\end{Proposition}
\begin{proof}
    All these properties follow immediately from the proofs of Proposition 4.9 and Proposition 4.10 in~\cite{hooks}. 
    In the proof of Proposition 4.9 the polynomial $Q(\z)$ in \cite{hooks} is assumed symmetric, 
    but this property is never used in the calculation completing the proof. A similar remark applies
    to Proposition~4.10.
\end{proof}

\begin{proposition}\label{p:pi-sym-restriction}
    By restriction, $\pi_{\lambda,N,a,b}$ induces an $\SL_2(\F)$-homomorphism
    \begin{align}
        \Sym_{|\lambda|+N}\Sym^b \E\otimes\Sym_\lambda\Sym^a E\to\Sym_N\Sym^b \E\otimes \Sym_\lambda\Sym^{a+b}\E
    \end{align}
    which is injective when $N\geq a$.
\end{proposition}
\begin{proof}
    By linearity, it suffices to consider a generic pure tensor of the domain:
    \[
    P(\x_{\lambda},\y_N)Q_{\lambda_1}(z_1,\ldots,z_{\lambda_1})\cdots Q_{\lambda_t}{(z_{|\lambda|-\lambda_t+1},\ldots, z_{|\lambda|})}
    \]
    for some symmetric polynomials $P,Q_{\lambda_1}, \ldots, Q_{\lambda_t}$, where $t=\ell(\lambda)$. By Definition~\ref{def:pi-map}, this element is mapped to
    \[
    P(\z_{|\lambda|},\y_N)Q_{\lambda_1}(z_1,\ldots,z_{\lambda_1})\cdots Q_{\lambda_t}{(z_{|\lambda|-\lambda_{t}+1},\ldots, z_{|\lambda|})},
    \]
    which is symmetric in each of the alphabets $\y, \{z_1,\ldots,z_{\lambda_1}\},\dots, \{z_{|\lambda|-\lambda_{t}+1},\ldots, z_{|\lambda|}\}$ separately. Moreover, each of the variables in $\y$ appears in each monomial with degree at most $b$, and each of the variables in $\z$ with degree at most $a+b$. Thus the restriction described above is well defined. Injectivity and $\SL_2(\F)$-equivariance are both inherited from the properties of $\pi$ proved in Proposition~\ref{p:injectivity pi}.
\end{proof}

We may abuse notation by also writing the restricted map as $\pi_{\lambda, N, a,b}$ or, when the 
context makes the parameters clear, just as $\pi$.

\begin{remark}
    When $\lambda=(M)$, the restricted map is precisely $\pi$  from \cite{hooks}.
\end{remark}

In the proof of Theorem~\ref{thm:main} we will make use of yet another restriction of $\pi$ which we postpone
until \S\ref{sec:tilde-pi}.

\subsubsection*{Two-row Weyl modules as vector spaces of polynomials}
In order to interpret Weyl modules in this context, let $\lambda=(n,m)$ to obtain
    $$
    \Sym_{(n,m)}\Sym^d \E \cong \Lambda_{\le d}[x_{11},\ldots,x_{1n}] \otimes 
    \Lambda_{\le d}[x_{21},\ldots,x_{2m}],
    $$ 
    which is spanned by pure tensors of symmetric polynomials
    $$f\otimes g =  f(x_{11},\ldots,x_{1n})g(x_{21},\ldots,x_{2m}).
    $$
    By Remark~\ref{remark:ExtProd}, the map $\psi_{(n,m)}$ from \eqref{eq:psi lam} corresponds to multiplication by the element
    $
    \prod_{i=1}^{m}\bigl( 1-(x_{1i},x_{2i}) \bigr)$
    of the group algebra $\F\SymG(x_{11},\ldots,x_{2m})$ of the symmetric group in the elements of $\x$. Thus, in light of Remark~\ref{rem:lower-exterior}, $\Delta^{(n,m)}\Sym^d \E$ is the submodule of $\bwedge{(n,m)}\Sym^d \E\le (\Sym^d E)^{\otimes(n+m)}$ spanned by the elements
    \[
    \prod_{i=1}^{m}(1-(x_{1i},x_{2i}))f(x_{11},\ldots,x_{1n})g(x_{21},\ldots,x_{2m})
    \]
    for symmetric $f$ and $g$.

\subsection{Weyl modules as quotients}\label{sec:weyl modules as quotients}
By the first isomorphism theorem, we can describe $\im(\psi_\lambda)$
as a quotient of $\Sym_\lambda \!V$ by a certain submodule $U_\lambda=\ker\psi_\lambda$. This is
done in \cite[2.1.15(a)]{Wey03}, in a construction that we present
here only for two-row partitions. (For an equivalent presentation see \cite[Theorem~1.2]{McDowell25}.) In the two-row
case,
$U_\lambda$ is defined as the submodule of $\Sym_\lambda V$ spanned by the images
$\im(\theta_{uv})$ of the following maps, for all choices of
$0\le u\le \lambda_1$ and $0\le v\le u+v < \lambda_2$. The map
$\theta_{uv}$ is defined as the composition
\[
\Sym_{(u, \lambda_1-u+\lambda_2-v,v)} \! V
\xrightarrow{1\otimes\symcoproduct\otimes1}  
\Sym_{(u,\lambda_1-u,\lambda_2-v,v)} \! V
\xrightarrow{\symproduct\otimes\symproduct}
\Sym_{(\lambda_1,\lambda_2)} \! V
\]
where, as defined in~\S\ref{subsec:productCoproduct}, $\symcoproduct$ is the coproduct for symmetric powers
and $\symproduct$ is the product for symmetric powers.

The following property of $U_{nm}$ is a consequence of Lemma~\ref{lem:commute}, and will be useful in the proof of Theorem~\ref{thm:main}. Recall that $\wr$ is the isomorphism from~\eqref{eq:wr}, the maps~$\extproduct$ and $\symcoproduct$ are the product for exterior powers and the coproduct for symmetric powers defined in \S\ref{subsec:productCoproduct}, and the map $\psi$ is the composition from Definition~\ref{def:weyl-as-image}. 
\begin{proposition}\label{p:wr-U_nn-ker-mu}
  For $0\le i \le m-1$, $k=m-i$, and $\tilde n=n-m+i=n-k$, we have $(\wr_k\otimes1)^{\otimes2}(\symcoproduct_{k\tilde n}\otimes \symcoproduct_{ki})U_{nm}\subseteq\ker(\extproduct_{(2^k)}\otimes\psi_{\tilde{n}i})$.
\end{proposition}
\begin{proof}
    The claim is that $U_{nm}$, that is the kernel of $\psi_{nm}$, is annihilated by $(\extproduct_{(2^k)}\otimes\psi_{\tilde{n}i})(\wr_k\otimes1)^{\otimes2}(\symcoproduct_{k\tilde n}\otimes \symcoproduct_{ki})$.
    We use Lemma~\ref{lem:commute} to write
    \begin{align*}
        (\extproduct_{(2^k)}\otimes\psi_{\tilde{n}i})(\wr_k\otimes1)^{\otimes2}&(\symcoproduct_{k\tilde n}\otimes \symcoproduct_{ki}) = 
        (\extproduct_{kk}\x_{1,*}^{\rho_k}\x_{2,*}^{\rho_k}\psi_{kk}\otimes \psi_{\tilde{n}i})(\symcoproduct_{k\tilde n}\otimes \symcoproduct_{ki})
        \\
        &= 
        (\extproduct_{kk}\x_{1,*}^{\rho_k}\x_{2,*}^{\rho_k}\otimes1)(\psi_{kk}\otimes \psi_{\tilde{n}i}) (\symcoproduct_{k\tilde n}\otimes \symcoproduct_{ki})
        \\
        &= 
        (\extproduct_{kk}\x_{1,*}^{\rho_k}\x_{2,*}^{\rho_k}\otimes1) \psi_{nm}\,.
    \end{align*}
    The claim is then clear.
\end{proof}

\section{Two auxiliary maps}\label{sec:aux}

In this section we define two $\SL_2(\F)$-homomorphisms which we will use in the proof of Theorem~\ref{thm:main},
namely
    \[
    \tilde\pi:
    \Sym_{n+m}\Sym^{d-k}\E\otimes
    \Delta^{(\tilde{n},i)}\Sym^k\E \hookrightarrow\Sym_{2k}\Sym^{d-k}\E\otimes\Delta^{(\tilde{n},i)}\Sym^d\E
    \]
and
\[
    \varphi: \Sym_{(n,m)}\Sym^k\E \twoheadrightarrow
    \Sym_{(\tilde{n},i)}\Sym^k\E,
\]
where $i\in\{0,1,\ldots,m-1\}$, $k=m-i$, and $\tilde n=n-m+i$.
We will show that they are an injection and a surjection, respectively. Both maps are of independent interest as plethystic lifts of certain $q$-binomial identities stated in~\cite{GMSW-FPSAC}.

\subsection{The \texorpdfstring{$\tilde\pi$}{π} map}\label{sec:tilde-pi}
Let $\lambda$ be a partition of size $M$.
In this subsection we prove, by means of an explicit field-independent $\SL_2(\F)$-homomorphism, that  whenever $N\geq a$
\begin{equation}\label{eq:SchurPositive}
\Qbinom{N+b}{b}\cdot s_{\lambda}\circ s_{a+b}(q,q^{-1})-\Qbinom{M+N+b}{b}\cdot s_\lambda\circ s_{a}(q,q^{-1})
\end{equation}
is Schur positive, meaning that it can be expressed as a non-negative linear combination of quantum integers.
When $\lambda = (M)$ and $N=a$, the difference becomes $0$ --- this is an instance of trinomial revision. Thus our proof follows and generalises the proof of~\eqref{eq:GTL} in~\cite{hooks}.

Towards this goal, we show that if $N\ge a$ then $\pi_{\lambda,N,a,b}$ restricts to an $\SL_2(\F)$-equivariant injection $\tilde{\pi}$ as follows (note that $\Delta^{\lambda}V\leq V^{\otimes \lambda}$ by Remark~\ref{rem:lower-exterior} and Definition~\ref{def:weyl-as-image}):
\[\begin{tikzcd}[column sep=4em]
	{\Sym_{M+N}\Sym^bE\otimes \Delta^\lambda\Sym^aE} & {\Sym_N\Sym^bE\otimes\Delta^\lambda\Sym^{a+b} E} \\
	{\Sym_{M+N}\Sym^b \E\otimes (\Sym^a\E)^{\otimes\lambda}} & {\Sym_{N}\Sym^b\E\otimes (\Sym^{a+b}\E)^{\otimes\lambda}}
	\arrow["{\tilde{\pi}}", dashed, from=1-1, to=1-2]
	\arrow[hook, from=1-1, to=2-1]
	\arrow[hook, from=1-2, to=2-2]
	\arrow["{\pi_{\lambda,N,a,b}}"', from=2-1, to=2-2]
\end{tikzcd}\]

As in Remark~\ref{remark:ExtProd}, consider the action of $\F\SymG_{M}$ on $(\Sym^\bullet\E)^{\otimes\lambda}$.

\begin{Lemma}\label{lem:mu-pi-commute}
    The action of $\F\SymG_{M}$ on the second tensor factors commutes with $\pi$.
\end{Lemma}
\begin{proof}
    By linearity, it suffices to verify the statement for $\sigma\in\SymG_{M}$ acting on variables $z_1,z_2,\ldots,z_{M}$. We compute directly:
    \begin{align*}
        \sigma\pi(P(\x_{M},\,\y_N)Q(\z_{M}))=\sigma\left(P(\z_{M},\y_N)Q(\z_{M})\right)=P(\z_{M},\y_N)\cdot\sigma Q(\z_{M})\\
        =\pi\left( P(\x_{M},\y_N)\cdot \sigma Q(\z_{M})\right)=\pi\sigma\left(P(\x_{M},\y_N)Q(\z_{M})\right).
    \end{align*}
    where the second and last equalities hold because $P$ (unlike $Q$) is a symmetric polynomial.
\end{proof}

Recall from~\ref{subsec:productCoproduct} 
that $\extproduct$ is the product on exterior powers and $\symcoproduct$ is the coproduct on symmetric powers.

\begin{corollary}
    Let $|\lambda|=M$ and $N\ge a$. Then, there is an injective $\mathrm{SL}_2(\F)$-homomorphism
    \[
    \tilde\pi:
    {\Sym_{M+N}\Sym^b \E\otimes \Delta^\lambda\Sym^a \E} \hookrightarrow {\Sym_N\Sym^b \E\otimes\Delta^\lambda\Sym^{a+b} \E}\,.
    \]
\end{corollary}
\begin{proof}
    We will first show that this restriction of $\pi$ is well-defined. Recall from \S\ref{sec:image} that 
    \[
    \Delta^{\lambda}\Sym^a\E = \extproduct^{\lambda}\symcoproduct^{\lambda} \Big(\Sym_{\lambda}\Sym^a\E\Big).
    \]
    Since $\symcoproduct^\lambda$ is a canonical inclusion, we view $\symcoproduct^\lambda\Sym_\lambda\Sym^a\E$ as an isomorphic copy of $\Sym_\lambda\Sym^a\E$ inside the space $(\Sym^a\E)^{\otimes\lambda}$ and so by Proposition~\ref{p:pi-sym-restriction},
    \begin{align*}
        \pi\Big(\Sym_{M+N}\Sym^b\E&\otimes \symcoproduct^\lambda\Sym_\lambda\Sym^a\E\Big)
        \\&\leq\Sym_{N}\Sym^{b}\E\otimes \symcoproduct^\lambda\Sym_\lambda\Sym^{a+b}\E\,.
    \end{align*}
    View $\extproduct^\lambda$ as an element of the group algebra $\F\SymG_{M}$ as in Remark~\ref{remark:ExtProd}.  
    By the observation above and Proposition~\ref{p:pi-sym-restriction},
    \begin{align*}
        \pi_{\lambda,N, a,b }\Big( &\Sym_{M+N}\Sym^{b}\E\otimes \Delta^{\lambda}\Sym^a\E \Big)\\
        &=\pi(1\otimes\extproduct^\lambda)\left(\Sym_{M+N}\Sym^{b}\E\otimes \symcoproduct^{\lambda}\Sym_{\lambda}\Sym^a\E\right)\\
        &=(1\otimes\extproduct^\lambda)\pi\left(\Sym_{M+N}\Sym^{b}\E\otimes \symcoproduct^{\lambda}\Sym_{\lambda}\Sym^a\E\right)\\
        &\leq (1\otimes\extproduct^{\lambda})\left(\Sym_{N}\Sym^{b}\E\otimes \symcoproduct^{\lambda}\Sym_{\lambda}\Sym^{a+b}\E\right)\\
        &=\Sym_{N}\Sym^{b}\E\otimes\Delta^{\lambda}\Sym^{a+b}\E\,.
    \end{align*}
    The $\SL_2(\F)$-equivariance and injectivity of $\tilde \pi$ is inherited from $\pi$.
\end{proof}

Now taking the characters on either side using~\eqref{eq:qPlethysm} we find
that~\eqref{eq:SchurPositive} is the character of the quotient of the codomain of $\tilde\pi$ by its domain,
and so, by~\eqref{eq:qCharacterSym}, it is a linear combination of quantum integers, as we 
claimed at the start of this subsection.

For ease of reference, we state the following special case.
\begin{corollary}\label{cor:tilde-pi}
    Suppose $\tilde{n}+k = n \ge m = i+k$. Then, there is an injective $\mathrm{SL}_2(\F)$-homomorphism
    \[
    \tilde\pi:
    \Sym_{n+m}\Sym^{d-k}\E\otimes
    \Delta^{(\tilde{n},i)}\Sym^k\E \hookrightarrow\Sym_{2k}\Sym^{d-k}\E\otimes\Delta^{(\tilde{n},i)}\Sym^d\E.
    \]
\end{corollary}

\subsection{The \texorpdfstring{$\varphi$}{φ} map}
We now prove by an explicit field-independent homomorphism of $\SL_2(\F)$-plethysms that 
\[
\Qbinom{n+k}{k}\Qbinom{m+k}{k}-\Qbinom{n}{k}\Qbinom{m}{k}
\]
is Schur positive.
By introducing two other variables $i$ and $\tilde{n}$ such that $k = m - i$ and $\tilde{n} = n-k$, the expression above becomes
\[
\Qbinom{n+k}{k}\Qbinom{m+k}{k}-\Qbinom{\tilde n+k}{k}\Qbinom{i+k}{k}\,.
\]
By a similar argument to the end of the previous subsection,
it suffices to construct a surjective $\SL_2(\F)$-homomorphism
\[
\varphi: \Sym_{(n,m)}\Sym^k\E \to
    \Sym_{(\tilde{n},i)}\Sym^k\E\,.
\]
The construction requires some combinatorial preliminaries.

\subsubsection*{The flip of a partition}

Recall from~\ref{subsec:partitions} that $L(k,k)$ is the set of partitions fitting into a $k \times k$ box.
Given a partition $\theta\in L(k,k)$, let $\flip_{k}\theta \subseteq [(k^k)]$ be the set of boxes obtained by reflecting the Young diagram $[\theta]$ by the main antidiagonal of $[(k^k)]$. Explicitly,
\[
    (r,c)\in\flip_k\theta\iff (k+1-c, k+1-r)\in[\theta].
\]
For instance,
if $\theta = (5,3,2,2,1)$ then we have
\ytableausetup{boxsize=.85em}
\[
    \theta = \ydiagram{5,3,2,2,1}, \quad
    \flip_5\theta = \ydiagram{5,5,5,5,5}*[\bullet]{4+1,4+1,3+2,1+4,5}\,.
\]
The complementary boxes $[(k^k)]\setminus\flip_k\theta$ can be read as the Young diagram of a partition;
in the example above this partition is $(4,4,3,1)$.

The next combinatorial result appears also in~\cite{AffineJT}.

\begin{lemma}\label{lem:flip} Given $\nu,\theta\in L(k,k)$, we have $(\nu+\rho_k) \sqcup (\theta+\rho_k) = \rho_{2k}$ if and only if $$[\nu]\sqcup\flip_k\theta = [(k^k)],$$
where $\sqcup$ denotes disjoint union.
\end{lemma}

\begin{proof}
We work by induction on $|\nu|$. When $\nu = \varnothing$ it is clear that 
$(\nu + \rho_k) \sqcup (\theta + \rho_k) = \rho_{2k}$ if and only if $\theta = (k^k)$,
and this is the case if and only if
$[\nu] \sqcup \flip_k\theta = [(k^k)]$. 
For the inductive step, let $(a,\nu_a) \in [\nu]$ be the position of the northern-most 
removable box. More formally, $a$ is maximal such that $\nu_1 = \cdots = \nu_a$.
Let
\begin{align*}
\nu^- = (\nu_1, \ldots, \nu_a-1, \ldots, \nu_k) 
\intertext{be the partition obtained from $\nu$ by removing this box from its Young diagram.
Let $b = k - \nu_a + 1$.
The diagram below, in which the boundary of $[\nu]$ is indicated by the thicker line, shows that $\nu \,\sqcup\, \flip_k\theta = [(k^k)]$ if and only if
$\flip_k\theta \,\cup\, \{ (a,\nu_a) \} = \flip_k\theta^+$, where}
{\theta^+} = (\theta_1, \ldots, \theta_b+1, \ldots, \theta_k). 
\end{align*}

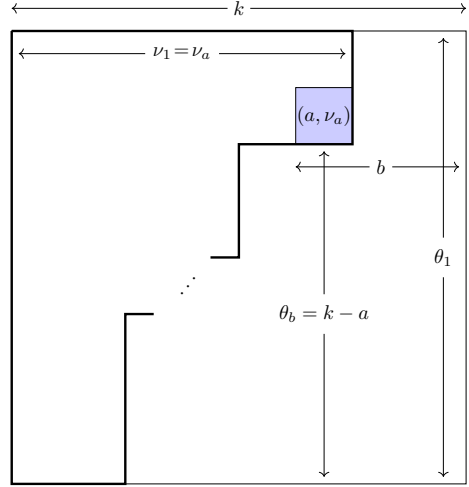
\begin{center}
\scalebox{0.75}{\begin{tikzpicture}[x=1cm, y=-1cm]
\draw (0,0)--(8,0)--(8,8)--(0,8)--cycle;
\fill[color = blue!20] (6,1)--(5,1)--(5,2)--(6,2)--cycle;
\draw (6,1)--(5,1)--(5,2); 
\draw[very thick] (0,0)--(6,0)--(6,2)--(4,2)--(4,4)--(3.5,4);
\draw[very thick] (2.5,5)--(2,5)--(2,8)--(0,8)--(0,0);

\node at (3,4.5) {\rotatebox{90}{$\ddots$}};
\node at (5.5,1.5) {$(a,\nu_a)$};

\node at (3,0.4) {$\nu_1 \!=\! \nu_a$};
\draw[->] (3.65,0.4)--(5.875,0.4);
\draw[->] (2.35,0.4)--(0.125,0.4);

\node at (5.5,5) {$\theta_b=k-a$};
\draw[->] (5.5,4.65)--(5.5,2.125);
\draw[->] (5.5,5.35)--(5.5,7.875);

\node at (7.6,4) {$\theta_1$};
\draw[->] (7.6,3.65)--(7.6,0.125);
\draw[->] (7.6,4.35)--(7.6,7.875);

\node at (6.5,2.4) {$b$};
\draw[->] (6.7,2.4)--(7.875,2.4);
\draw[->] (6.3,2.4)--(5,2.4);

\node at (4,-0.4) {$k$};
\draw[->] (4.2,-0.4)--(8,-0.4);
\draw[->] (3.8,-0.4)--(0,-0.4);

\end{tikzpicture}}
\end{center}
 
\noindent We now have a chain of double implications. The second double implication 
is the inductive hypothesis for $\nu^-$, while all the rest are simple rewritings,
using $\nu_a = k-b+1$ and $\theta_b = k-a$ and that $\rho_k = (k-1,\ldots, 1,0)$:
\begin{align*}
[\nu] &\sqcup \flip_k\theta = [(k^k)] \\ 
&\iff [\nu^-] \sqcup \flip_k\theta^+ = [(k^k)] \\
&\iff (\nu^- + \rho_k) \sqcup (\theta^+ + \rho_k) = \rho_{2k} \\
&\iff (\nu_1 + k-1, \ldots, \nu_a-1+k-a, \ldots, \nu_k) \\
&\phantom{~\iff~}\quad \sqcup 
      (\theta_1 + k - 1, \ldots, \theta_b + 1 + k - b, \ldots, \theta_k) = \rho_{2k} \\
&\iff (\nu_1 + k-1, \ldots, 2k-a-b, \ldots, \nu_k) \\
&\phantom{~\iff~}\quad \sqcup 
      (\theta_1 + k - 1, \ldots, 2k-a-b+1, \ldots, \theta_k) = \rho_{2k}  \\
&\iff (\nu_1 + k-1, \ldots, 2k-a-b+1, \ldots, \nu_k) \\
&\phantom{~\iff~}\quad \sqcup 
      (\theta_1 + k - 1, \ldots, 2k-a-b, \ldots, \theta_k) = \rho_{2k} \\
&\iff (\nu_1 + k-1, \ldots, \nu_a + k -a, \ldots, \nu_k) \\
&\phantom{~\iff~}\quad \sqcup 
      (\theta_1 + k - 1, \ldots, \theta_b + k -b, \ldots, \theta_k) = \rho_{2k} \\
&\iff (\nu + \rho_k) \sqcup (\theta + \rho_k) = \rho_{2k}.
\end{align*}
This completes the inductive step. 
\end{proof}

\begin{corollary}\label{cor:flip-alternating-poly}
    Let $\nu,\theta\in L(k,k)$. Then
    \[
    a_{(\nu+\rho_k)\sqcup(\theta+\rho_k)}=\begin{cases}
        a_{\rho_{2k}}&\text{when $[\nu]\sqcup\flip_k\theta=[(k^k)]$}\\
        0&\text{otherwise.}
    \end{cases}
    \]
\end{corollary}
\begin{proof}
    Note that $(\nu+\rho_k)\sqcup(\theta+\rho_k)\in L(2k, 2k-1)$. The alternating polynomial $a_{(\nu+\rho_k)\sqcup(\theta+\rho_k)}$ is non-zero if and only if $(\nu+\rho_k)\sqcup(\theta+\rho_k)$ has all parts distinct, which is the case only when it equals $\rho_{2k}$. By Lemma~\ref{lem:flip}, this is equivalent to $[\nu]\sqcup\flip_k\theta=[(k^k)]$.
\end{proof}

\subsubsection*{Construction and surjectivity of $\varphi$}
By the polynomial interpretation of one-row Weyl modules, we consider the map~$\varphi$ constructed as
\begin{align}\label{eq:def-phi}
    \varphi:\quad \Sym_{(n,m)}\Sym^k\E &\to
    \Sym_{(\tilde{n},i)}\Sym^k\E \notag\\
    \Lambda_{\le k}[\x_n]\otimes \Lambda_{\le k}[\z_m] &\to 
    \Lambda_{\le k}[\y_{\tilde{n}}]\otimes \Lambda_{\le k}[\mathbf{w}_i] \\ 
    s_\lambda\otimes s_\mu 
    &\mapsto 
    \sum_{\substack{
    \nu,\theta \colon [\nu]\hskip1pt\sqcup\hskip1pt\flip_k\theta = [(k^k)]\notag\\
    \nu \subseteq \lambda \subseteq \nu\hskip1pt\sqcup\hskip1pt(k^{\tilde{n}})\notag\\
    \theta \subseteq \mu \subseteq \theta\hskip1pt\sqcup\hskip1pt(k^i)
    }}
    \sgn_{kk}(\nu,\theta)s_{\lambda/\nu} \otimes s_{\mu / \theta},
\end{align}
where $\sgn_{kk}(\nu,\theta)$ is defined in Remark~\ref{rmk:ext-product}. 
As we shall see, by following the computations in the proof of Lemma~\ref{lem:phi-surj-in-thm1.1}, the map $\varphi$ arises as the composition $$(\wr_{2k}^{-1}\otimes 1)(\extproduct_{(2^k)}\otimes 1)(\wr_k\otimes 1)^{\otimes 2}(\symcoproduct_{k\tilde{n}}\otimes\symcoproduct_{ki}),$$ thus it is a well-defined $\SL_2(\F)$-homomorphism.
(Here $\extproduct$ and $\symcoproduct$ are the product on exterior powers and coproduct on symmetric powers defined in~\S\ref{subsec:productCoproduct}.) 
To show surjectivity we establish the stronger property that the restriction to $\Sym_{(n,i)}\Sym^k\E$ is surjective.

\begin{remark}\ytableausetup{boxsize=.5em}
    There are many partitions involved in the definition of the image of $\varphi$. They fit nicely in the following puzzle, which we illustrate with an example for $k=5$, $i=2$, and $\tilde{n} = 4$. We begin with two partitions $\nu = (4,4,3,1)$ and $\theta = (5,3,2,2,1)$ such that $\nu\sqcup\flip_5\theta=[(5^5)]$:
    \[ 
    \ytableausetup{boxsize=.85em}
    \ydiagram{5,5,5,5,5}*[*(lightgray!50)]{4+1,4+1,3+2,1+4,5}*[*(gray)]{4,4,3,1}\,.
    \]
    Next, we consider a partition $\lambda \in L(n,k)$ such that 
    $\nu \subseteq \lambda \subseteq \nu\sqcup(k^{\tilde{n}}) = \nu\sqcup(5^4)$, which means that
    $[\lambda/\nu]$ fits in this smaller area marked with $\ytableaushort{\cdot}$s below, in which every column has height $\tilde{n} = 4$:
    \[
    \ydiagram{5,5,5,5,5,5,5}*[\cdot]{4+1,4+1,3+2,1+4,4,4,3,1}*[*(gray)]{4,4,3,1}*[*(lightgray!50)]{0,0,0,0,4+1,4+1,3+2,1+4,5}\,.
    \]
    On the flip side, we have $\theta \subseteq \mu \subseteq \theta\sqcup(k^i) = \theta\sqcup(5^2)$. And so we conclude that the reflection of $[\mu/\theta]$ along the main antidiagonal of $[(k^k)]$ (which we denote $\flip_k(\mu/\theta)$ for consistency) fits into the following area marked with $\ytableaushort{\bullet}$s, 
    in which every row has width $i = 2$:
    \[
    \ydiagram{7,7,7,7,7,2+5,2+5}*[\cdot]{6+1,6+1,5+2,3+4,2+4,2+4,2+3,2+1}*[\bullet]{4+2,4+2,3+2,1+2,2}*[*(gray)]{4,4,3,1}*[*(lightgray!50)]{0,0,0,0,6+1,6+1,5+2,3+4,2+5}\,.
    \]
    In summary: the $\ydiagram[*(gray)]{1}$ area is (a translation of) $[\nu]$, the $\ydiagram[*(lightgray!50)]{1}$ area is (a translation of) $\flip_k\theta$, the $\ytableaushort{\cdot}$ area is a super-set of $[\lambda/\nu]$ and the $\ytableaushort{\bullet}$ area is a super-set of $\flip_k(\mu/\theta)$. This diagram depends only on $n, i, \tilde{n}$ and $\nu$. 
\end{remark}

In the following lemma we see 
\(\Lambda_{\le k}[\z_i] = \langle s_\tau \mid \tau \in L(i,k)\rangle\)
as a subspace of $\Lambda_{\le k}[\z_m] = \langle s_\tau \mid \tau \in L(m,k)\rangle$.
\begin{lemma}\label{lem:surjective}
    The restriction of $\varphi$ to
    \begin{align*}
    \Lambda_{\le k}[\x_n]\otimes \Lambda_{\le k}[\z_i] &\to 
    \Lambda_{\le k}[\y_{\tilde{n}}]\otimes \Lambda_{\le k}[\mathbf{w}_i] \\ 
    s_\lambda\otimes s_\mu 
    &\mapsto 
    \sum_{\substack{
    \nu,\theta \colon [\nu]\hskip1pt\sqcup\hskip1pt\flip_k\theta = [(k^k)]\\
    \nu \subseteq \lambda \subseteq \nu\hskip1pt\sqcup\hskip1pt(k^{\tilde{n}})\\
    \theta \subseteq \mu \subseteq \theta\hskip1pt\sqcup\hskip1pt(k^i)
    }}
    \sgn_{kk}(\nu,\theta)s_{\lambda/\nu} \otimes s_{\mu / \theta}
    \end{align*}
    is a surjective linear map.
\end{lemma}
\begin{proof}
    We proceed by induction on $\mu \in L(i,k)^{\preceq}$, where the ordering $\preceq$ is lexicographic, to show that the restriction of $\varphi$ to
    \[
    \Lambda_{\le k}[\x_n] \otimes \langle s_{\tau} \mid \tau \in L(i,k), ~\tau \preceq \mu\rangle
    \twoheadrightarrow
    \Lambda_{\le k}[\y_{\tilde{n}}] \otimes \langle s_{\tau} \mid \tau \in L(i,k), ~\tau \preceq \mu\rangle,
    \]
    is surjective. For $\mu=(k^i)$ this gives the desired conclusion.
    
    When $\mu=\varnothing$, we have only one pair $(\nu, \theta)$ in the sum, namely $\nu = (k^k)$ and $\theta=\varnothing$ with $\sgn_{kk}((k^k),\varnothing)=1$.
    Thus,
    \[
    \varphi(s_\lambda \otimes 1)= s_{\lambda / (k^k)} \otimes 1.
    \]
    Let $\lambda$ range over $\{\lambda\in L(n,k) \mid (k^k)\subseteq \lambda\}$. 
    Each such partition can be written as $\lambda = \pi \sqcup (k^k)$ for some $\pi\in L(\tilde{n},k)$. We have $s_{\lambda / (k^k)} = s_\pi$. Therefore, the image of $\Lambda_{\le k}[\x_n]\otimes\F$ is  $\Lambda_{\le k}[\y_{\tilde{n}}]\otimes \F$.

    For the induction step, suppose that
    \[
    \Lambda_{\le k}[\x_n] \otimes \langle s_{\tau} \mid \tau \in L(i,k), ~\tau \prec \mu\rangle
    \twoheadrightarrow
    \Lambda_{\le k}[\y_{\tilde{n}}] \otimes \langle s_{\tau} \mid \tau \in L(i,k), ~\tau \prec \mu\rangle
    \]
    and we want to show that the same holds if the symbol $\prec$ is changed for $\preceq$ in both the domain and codomain.
    For this, consider the image of an element $s_\lambda\otimes s_\mu$ for a partition $\lambda$ of the form $\pi \sqcup (k^k)$ as above:
    \begin{align}\label{eq:s_pi-s_mu}
    s_{\pi \sqcup (k^k)}\otimes s_\mu \mapsto
    s_{\pi} \otimes s_\mu +
    \sum_{(\nu,\theta)\ne((k^k),\varnothing)} \sgn_{kk}(\nu,\theta)s_{(\pi \sqcup (k^k))/\nu} \otimes s_{\mu/\theta},
    \end{align}
    where, again, the coefficient at the leading term is $1=\sgn_{kk}((k^k),\varnothing)$.
    We shall now show that $\sum_{(\nu,\theta)\ne((k^k),\varnothing)} \sgn_{kk}(\nu,\theta)s_{\lambda/\nu} \otimes s_{\mu/\theta}$ is in $$\Lambda_{\le k}[\y_{\tilde{n}}] \otimes \langle s_{\tau} \mid \tau \in L(i,k), ~\tau \prec \mu\rangle.$$ 
    Since $(\pi \sqcup (k^k))/\nu\subseteq(\nu\sqcup(k^{\tilde{n}}))/\nu$ we have that $s_{(\pi \sqcup (k^k))/\nu}$ is a symmetric polynomial in $\tilde{n}$ variables, each of degree at most $k$. Therefore,
    $s_{(\pi \sqcup (k^k))/\nu}\in\Lambda_{\leq k}[\y_{\tilde{n}}]$. 
    On the other hand, write $s_{\mu/\theta} = \sum_\tau c^\mu_{\theta\tau}s_\tau$. If the coefficient $c^\mu_{\theta\tau}$ is nonzero, then $\tau$ is contained in $\mu$ and therefore $\tau \preceq \mu$. Moreover, since $\theta\neq\varnothing$, we have $|\tau|=|\mu|-|\theta|<|\mu|$, and hence $\tau\prec\mu$. Hence
    \(
    s_{\mu/\theta}\in\langle s_{\tau} \mid \tau \in L(i,k), ~\tau \prec \mu\rangle.
    \)
    Thus, by induction, the sum in~\eqref{eq:s_pi-s_mu} is in the image of $\varphi$, and consequently so is $s_\pi\otimes s_\mu$.
    Altogether we have shown that for any fixed $\mu$
    \[
    \langle s_\lambda \otimes s_\mu  \mid \lambda\in L(n,k), ~ (k^k)\subseteq \lambda \rangle \to 
    \Lambda_{\le k}[\y_{\tilde n}] \otimes \langle s_{\tau} \mid \tau \in L(i,k), ~\tau \preceq \mu\rangle.
    \]
    We conclude
    \[
    \Lambda_{\le k}[\x_n] \otimes \langle s_{\tau} \mid \tau \in L(i,k), ~\tau \preceq \mu\rangle
    \twoheadrightarrow
    \Lambda_{\le k}[\y_{\tilde{n}}] \otimes \langle s_{\tau} \mid \tau \in L(i,k), ~\tau \preceq \mu\rangle,
    \]
    as claimed.
\end{proof}

Since the restriction of $\varphi$ above is surjective, and has the same codomain as $\varphi$, we immediately deduce the following corollary. 

\begin{corollary}\label{cor:phi-surj}
    The map $\varphi$ defined in \eqref{eq:def-phi} is surjective.
\end{corollary}
\begin{proof}
    Since $\Sym_{(n,i)}\Sym^k \E\subseteq\Sym_{(n,m)}\Sym^k \E$ (as vector spaces, but not necessarily as modules), by Lemma~\ref{lem:surjective} we have
    \[
        \Sym_{(\tilde n,i)}\Sym^k E=\varphi(\Sym_{(n,i)}\Sym^k E)\subseteq\im\varphi\subseteq \Sym_{(\tilde n,i)}\Sym^k E,
    \]
    and so $\im\varphi$ equals the codomain of $\varphi$ as vector spaces, thus $\varphi$ is surjective.
\end{proof}

\section{The proof of Theorem\texorpdfstring{~\ref{thm:main}}{ 1.1}}\label{sec:proof}
\subsection{Set-up and notation}\label{sec:set-up} In this section we assume the following notation: for any fixed $i=0,1,\ldots, m-1$, let $k=m-i$ and $\tilde{n}=n-m+i=n-k$.

Theorem~\ref{thm:main} is equivalent to the existence of nested $\SL_2(\F)$-modules
\begin{equation}\label{eq:chain_Mi}
0=M_m\leq M_{m-1}\leq\dots\leq M_0=\Delta^{(n,m)}\Sym^d \E
\end{equation}
which satisfy the following diagram
of short exact sequences:
\[\begin{tikzcd}
	& {M_m} & {M_{m-1}} & {L_{m-1}} \\
	& {M_{m-2}} & \cdots \\
	{L_{m-2}} & {M_2} & {M_1} & {L_1} \\
	{L_2} & {M_0} \\
	{L_0}
	\arrow[hook', from=1-2, to=1-3]
	\arrow[two heads, from=1-3, to=1-4]
	\arrow[hook, from=1-3, to=2-2]
	\arrow[hook', from=2-2, to=2-3]
	\arrow[two heads, from=2-2, to=3-1]
	\arrow[from=2-3, to=3-2]
	\arrow[hook', from=3-2, to=3-3]
	\arrow[two heads, from=3-2, to=4-1]
	\arrow[two heads, from=3-3, to=3-4]
	\arrow[hook, from=3-3, to=4-2]
	\arrow[two heads, from=4-2, to=5-1]
\end{tikzcd}\]
where $L_i = \Sym_{n+m}\Sym^{d-k}\E\otimes\Delta^{(\tilde{n},i)}\Sym^{k}\E$ are the filtration layers from Theorem~\ref{thm:main}.

In order to define the modules $M_i$, we first construct maps 
\[
\mathcal{M}_i : M_i \to \tilde{L}_i
\]
where for $0\le i \le m-1$ we define
\[
\tilde L_i= \Sym_{2k}\Sym^{d-k}\E\otimes\Delta^{(\tilde{n},i)}\Sym^d\E.
\] 
In a later section, we show that the image of $M_i$ lands inside an isomorphic copy of $L_i$ inside $\tilde{L}_i$ (see \S\ref{subsec:tilde-L}).
These maps $\mathcal{M}_i$ are given by the diagram in Figure~\ref{big diagram}; we now proceed to construct them.
Consider the map $\what{\mathcal{M}}_i:\Sym_{(n,m)}\Sym^d E\to \tilde L_i$ given by a composition
\[
\what{\mathcal{M}}_i=(\wr_{2k}^{-1}\otimes 1)(\extproduct_{(2^k)}\otimes \psi_{\tilde ni})(\wr_k\otimes 1)^{\otimes 2}(\symcoproduct_{k\tilde n}\otimes \symcoproduct_{ki}).
\]
Here the maps $\extproduct$ and $\symcoproduct$ are the product for exterior powers and
the coproduct for symmetric powers defined in \S\ref{sec:tensor spaces}, the map $\psi$ in \S\ref{sec:image} and the map $\wr$ is the isomorphism from \eqref{eq:wr as poly}. All of the above are $\SL_2(\F)$-equivariant, hence so is $\what{\mathcal{M}}_i$. By Proposition~\ref{p:wr-U_nn-ker-mu}, $\what{\mathcal{M}}_i(U_{nm})=0$, so by definition of Weyl modules from \S\ref{sec:weyl modules as quotients},~$\what{\mathcal{M}}_i$ descends to the homomorphism 
\[
\overline{\mathcal{M}}_i:\Delta^{(n,m)}\Sym^dE\to \tilde L_i.
\]

The modules $M_i$ and homomorphisms $\mathcal{M}_i$ are defined inductively, by setting $M_0=\Delta^{(n,m)}\Sym^d E$, and for $0\le i \le m-1$:
\[
\mathcal{M}_i=\overline{\mathcal{M}}_i|_{M_i}\qquad \text{and} \qquad  M_{i+1}=\ker(\mathcal{M}_{i}).
\]

\smallskip
\begin{figure}
\hspace*{-0.3in}\begin{adjustbox}{scale=1.2}
    \begin{tikzcd}
        {\scriptstyle M_{i}}
         \\
        {
        \begin{tikzpicture}[x=.45cm,y=.45cm,baseline=.5cm]
        \draw (0,1) rectangle ++ (5,-1);
        \draw (0,2) rectangle ++ (7,-1);
            \node(13) at (0.5, .5){\tiny$y_{21}$};
            \node(14) at (1.5, .5){\tiny$\cdots$};
            \node(14) at (2.5, .5){\tiny$y_{2k}$};
            \node(14) at (3.5, .5){\tiny$\cdots$};
            \node(14) at (4.5, .5){\tiny$y_{2m}$};
            \node(13) at (0.5, 1.5){\tiny$y_{11}$};
            \node(14) at (1.5, 1.5){\tiny$\cdots$};
            \node(14) at (2.5, 1.5){\tiny$y_{1k}$};
            \node(14) at (3.5, 1.5){\tiny$\cdots$};
            \node(14) at (4.5, 1.5){\tiny$y_{1m}$};
            \node(14) at (5.5, 1.5){\tiny$\cdots$};
            \node(14) at (6.5, 1.5){\tiny$y_{1n}$};
          \end{tikzpicture}_{{\hspace{-2.85em}\scriptscriptstyle d\hspace{2.2em}}} / \scriptstyle  U_{nm}} &
        \begin{tikzpicture}[x=.45cm,y=.45cm,baseline=.2cm]
        \draw 
            (0,1) rectangle ++ (3,-1);
            \node(13) at (0.5, .5){\tiny$x_{11}$};
            \node(14) at (1.5, .5){\tiny$\cdots$};
            \node(14) at (2.5, .5){\tiny$x_{2k}$};
        \end{tikzpicture}_{\!\!\scriptscriptstyle d-k}
        \scriptstyle \otimes 
        \Delta^{(\tilde{n},i)}\Sym^d\E
        \\
        {
        \begin{tikzpicture}[x=.45cm,y=.45cm,baseline=.5cm]
        \draw (0,1) rectangle ++ (3,-1);
        \draw (0,2) rectangle ++ (3,-1);
            \node(13) at (0.5, 1.5){\tiny$x_{11}$};
            \node(14) at (1.5, 1.5){\tiny$\cdots$};
            \node(14) at (2.5, 1.5){\tiny$x_{1k}$};
            \node(13) at (0.5, .5){\tiny$x_{21}$};
            \node(14) at (1.5, .5){\tiny$\cdots$};
            \node(14) at (2.5, .5){\tiny$x_{2k}$};
        \end{tikzpicture}_{\!\!\scriptscriptstyle d}
        \begin{tikzpicture}[x=.45cm,y=.45cm,baseline=.5cm]
        \draw (0,1) rectangle ++ (3,-1);
        \draw (0,2) rectangle ++ (5,-1);
            \node(13) at (0.5, 1.5){\tiny$y_{11}$};
            \node(14) at (1.5, 1.5){\tiny$\cdots$};
            \node(14) at (2.5, 1.5){\tiny$y_{1i}$};
            \node(14) at (3.5, 1.5){\tiny$\cdots$};
            \node(14) at (4.5, 1.5){\tiny$y_{1\tilde{n}}$};
            \node(13) at (0.5, .5){\tiny$y_{21}$};
            \node(14) at (1.5, .5){\tiny$\cdots$};
            \node(14) at (2.5, .5){\tiny$y_{2i}$};
        \end{tikzpicture}_{{\hspace{-2.85em}\scriptscriptstyle d\hspace{2.2em}}} /\scriptstyle U_{nm}
        } & {
        \begin{tikzpicture}[x=.45cm,y=.45cm,baseline=-.7cm,rotate=-90]
        \draw (0,1) rectangle ++ (3,-1);
            \node(13) at (0.5, .5){\tiny$x_{11}$};
            \node(14) at (1.8, .5){\tiny$\vdots$};
            \node(14) at (2.5, .5){\tiny$x_{2k}$};
        \end{tikzpicture}_{\!\!\scriptscriptstyle d+k-1}}
       \scriptstyle \hspace{-2em}\otimes
        \Delta^{(\tilde{n},i)}\Sym^d\E
        \\
        {        
        \begin{tikzpicture}[x=.45cm,y=.45cm,baseline=-.4cm,rotate=-90]
        \draw (0,1) rectangle ++ (3,-1);
        \draw (0,2) rectangle ++ (3,-1);
            \node(13) at (0.5, .5){\tiny$x_{11}$};
            \node(14) at (1.8, .5){\tiny$\vdots$};
            \node(14) at (2.5, .5){\tiny$x_{1k}$};
            \node(13) at (0.5, 1.5){\tiny$x_{21}$};
            \node(14) at (1.8, 1.5){\tiny$\vdots$};
            \node(14) at (2.5, 1.5){\tiny$x_{2k}$};
        \end{tikzpicture}_{\!\!\scriptscriptstyle d+k-1}\hspace{-2.7em}
        \begin{tikzpicture}[x=.45cm,y=.45cm,baseline=.5cm]
        \draw (0,1)--(3,1)--(3,0)--(0,0); 
        \draw (0,2)--(5,2)--(5,1)--(0,1); 
            \node(13) at (0.5, 1.5){\tiny$y_{11}$};
            \node(14) at (1.5, 1.5){\tiny$\cdots$};
            \node(14) at (2.5, 1.5){\tiny$y_{1i}$};
            \node(14) at (3.5, 1.5){\tiny$\cdots$};
            \node(14) at (4.5, 1.5){\tiny$y_{1\tilde{n}}$};
            \node(13) at (0.5, .5){\tiny$y_{21}$};
            \node(14) at (1.5, .5){\tiny$\cdots$};
            \node(14) at (2.5, .5){\tiny$y_{2i}$};
        \end{tikzpicture}_{{\hspace{-2.85em}\scriptscriptstyle d\hspace{2.2em}}}
        \hspace{5em}
        \begin{tikzpicture}[remember picture, overlay]
            \node (*) at (-2.8em,-1.5em) {$\scriptstyle /(\wr_k\otimes 1)^{\scriptscriptstyle \otimes 2}U_{nm}$};
        \end{tikzpicture}
        }
        & {
        \begin{tikzpicture}[x=.45cm,y=.45cm,baseline=-.4cm,rotate=-90]
        \draw (0,1) rectangle ++ (3,-1);
        \draw (0,2) rectangle ++ (3,-1);
            \node(13) at (0.5, .5){\tiny$x_{11}$};
            \node(14) at (1.8, .5){\tiny$\vdots$};
            \node(14) at (2.5, .5){\tiny$x_{1k}$};
            \node(13) at (0.5, 1.5){\tiny$x_{21}$};
            \node(14) at (1.8, 1.5){\tiny$\vdots$};
            \node(14) at (2.5, 1.5){\tiny$x_{2k}$};
        \end{tikzpicture}_{\scriptscriptstyle \!\!d+k-1}\hspace{-2.7em}
        \begin{tikzpicture}[x=.45cm,y=.45cm,baseline=.5cm]
        \draw (0,1)--(3,1)--(3,0)--(0,0); 
        \draw (0,2)--(5,2)--(5,1)--(0,1); 
            \node(13) at (0.5, 1.5){\tiny$y_{11}$};
            \node(14) at (1.5, 1.5){\tiny$\cdots$};
            \node(14) at (2.5, 1.5){\tiny$y_{1i}$};
            \node(14) at (3.5, 1.5){\tiny$\cdots$};
            \node(14) at (4.5, 1.5){\tiny$y_{1\tilde{n}}$};
            \node(13) at (0.5, .5){\tiny$y_{21}$};
            \node(14) at (1.5, .5){\tiny$\cdots$};
            \node(14) at (2.5, .5){\tiny$y_{2i}$};
        \end{tikzpicture}_{{\,\scriptscriptstyle\hspace{-3.05em}d\hspace{2.2em}}} 
        \hspace{5em}
        \begin{tikzpicture}[remember picture, overlay]
            \node (*) at (-2.6em,-1.5em) {$\scriptstyle /\ker(\extproduct_{\scriptscriptstyle (2^k)}\otimes\,\psi_{\tilde n i})$};
        \end{tikzpicture}
        }
        \arrow[hook', from=1-1, to=2-1]
        \arrow["\symcoproduct_{k\tilde{n}}\otimes\symcoproduct_{ki}"', hook', from=2-1, to=3-1]
        \arrow["{(\wr_{k}\otimes1)^{\otimes2}}"'{inner sep=.8ex}, two heads, hook', from=3-1, to=4-1]
        \arrow["{\wr_{2k}^{-1}\otimes1}"'{inner sep=.8ex}, two heads, hook, from=3-2, to=2-2]
        \arrow["\extproduct_{(2^k)}\otimes\psi_{\tilde{n}i}"'{inner sep=.8ex}, two heads, hook, from=4-2, to=3-2]
        \arrow[two heads, from=4-1, to=4-2]
        \arrow["\scriptstyle \mathcal{M}_i", dashed, from=1-1, to=2-2]
    \end{tikzcd}
\end{adjustbox}
\caption{The restriction $\mathcal{M}_i$ of the composition $\hmi$ to $M_i$. The rectangular shapes introduced in \S\ref{sec:tensor spaces} stand for symmetric and exterior powers. The maps $\symcoproduct$ and $\extproduct$ are the coproduct of symmetric powers and the product of exterior powers defined in \S\ref{subsec:productCoproduct}, the map $\psi$ is the composition from Definition~\ref{def:weyl-as-image}.
We write simply 
$U_{nm}$ for $(\symcoproduct_{k\tilde n}\otimes\symcoproduct_{ki})U_{nm}$
(see Remark~\ref{remark:ExtProd}).
}
\label{big diagram}
\end{figure}

\medskip\noindent

\subsection{Proof outline}
We pause here to assess what remains to complete the proof. In the previous section we defined a nested sequence of $\SL_2(\F)$-modules 
\[
M_m\le M_{m-1}\le \dots \le M_0=\Delta^{(n,m)}\Sym^d E.
\]
We will show that $M_m=0$ in the concluding dimension counting argument in \S\ref{sec:dimension}.

As we pointed out, we need to show that for $i=0,1,\ldots,m-1$ there exists a short exact sequence
\[
0\to M_{i+1}\to M_i\to L_i\to 0.
\]
Note that $\tilde\pi_{(\tilde n,i), 2k, k,d-k}$ has domain $L_i$ and codomain $\tilde L_i$, so by Corollary~\ref{cor:tilde-pi},
\[
L_i\cong \tilde\pi(L_i)\leq \tilde L_i,
\]
and, by definition, $M_{i+1}=\ker(\mathcal{M}_i)$. To conclude the proof of Theorem~\ref{thm:main} it therefore remains to show that $\im(\mathcal{M}_i)=\tilde\pi(L_i)\cong L_i$. Indeed, this will show that the short exact sequence
\[
0\to\ker(\mathcal{M}_i)\to M_i\to \im(\mathcal{M}_i)\to 0
\]
is precisely
\[
0\to M_{i+1}\to M_i\to L_i\to 0.
\]
We will show that $\im\mathcal{M}_i=\tilde\pi(L_i)$ in two steps, first arguing that $\tilde\pi(L_i)$ is contained in the image of $\mathcal{M}_i$ (viewed just as a linear map of $\F$-vector spaces) and then concluding by dimension counting.

\subsection{\texorpdfstring{$\tilde\pi(L_i)$}{π(Li)} is a subspace of \texorpdfstring{$\im\mathcal{M}_i$}{im Mi}}
\label{subsec:tilde-L}
In this subsection we turn to consider $\mathcal{M}_i$ primarily as a linear map to deduce that $\mathcal{M}_i(M_i)$ \emph{contains} $\tilde\pi(L_i)$. 
Recall the $\pi$ map from Definition~\ref{def:pi-map} and consider the $\SL_2(\F)$-module
\[
S_i=(\pi_{(n,m),0,k,d-k}(A_i\otimes B_i)+U_{nm})/U_{nm}\le \Delta^{(n,m)}\Sym^d E,
\]
where
\[
A_i = 
\Sym_{n+m}\Sym^{d-k}\E
\quad\text{and}\quad
B_i = 
\Sym_{(n,m)}\Sym^k\E\,.
\]
\begin{lemma}
    $S_i$ is a submodule of $M_i$.
\end{lemma}
\begin{proof}
By definition, $S_i$ is a submodule of $M_0=\Delta^{(n,m)}\Sym^d E$. Suppose that for $j<i$, $S_i$ is a submodule of $M_j$ and we will show that $S_i$ is a submodule of $M_{j+1}$. The claim then follows by induction. Let $k'=m-j$ and $\tilde n'=n-m+j=n-k'$. Note that $k<k'$ because $j<i$.

Recall that $M_{j+1}=\ker(\mathcal{M}_{j})$ and $\mathcal{M}_{j}$ is
obtained by reducing $\what{\mathcal{M}}_{j}$ to $\overline{\mathcal{M}}_j$ and then restricting to $M_{j}$.
It suffices then to show that $\pi(A_i\otimes B_i)$ is annihilated by
\[
\what{\mathcal{M}}_j=(\extproduct_{(2^{k'})}\otimes\psi_{(\tilde{n}',j)})
(\wr_{k'}\otimes1)^{\otimes2}
(\symcoproduct_{(k',\tilde{n}')}\otimes\symcoproduct_{(k',j)})\,,
\] 
where, as defined in \S\ref{subsec:productCoproduct}, $\extproduct$ is the product on exterior powers and
$\symcoproduct$ is the coproduct on symmetric powers.
The Wronskian $\wr$ is given by multiplication by an alternating polynomial in the variables of the second tensor factor; since $\pi$ acts as an evaluation homomorphism, this multiplication commutes with $\pi$. Together with Lemma~\ref{lem:mu-pi-commute} we obtain that the whole composition commutes with $\pi$:
\[
\what{\mathcal{M}}_j\pi(A_i\otimes B_i)=\pi(A_i\otimes \what{\mathcal{M}}_j(B_i)).
\]
We encourage the reader to recall Definition~\ref{def:pi-map} and the proof of Lemma~\ref{lem:mu-pi-commute} to see that in the second line above, the composition indeed acts only on $B_i$, leaving~$A_i$ unchanged. Indeed, $B_i$ is an $\F$-vector subspace of $\Sym_{(n,m)}\Sym^dE$, the domain of \smash{$\what{\mathcal{M}}_j$} (which is best seen through the Schur basis in polynomial interpretation), so the composition \smash{$\what{\mathcal{M}}_j$} is well-defined on the elements of $B_i$. Expanding the second tensor factor we compute explicitly (again, we refer to diagram~\eqref{big diagram} to keep track of the tensor products): 
\begin{multline*}
(\extproduct_{(2^{k'})}\otimes\psi_{\tilde{n}'j})
(\wr_{k'}\otimes1)^{\otimes2}
(\symcoproduct_{k'\tilde{n}'}\otimes\symcoproduct_{k'j})B_i
\\
\leq\extproduct_{(2^{k'})}\wr_{k'}^{\otimes 2}\Sym_{(k', k')}\Sym^k E \otimes \psi_{\tilde n' j}\Sym_{(\tilde n',j)}\Sym^k E\\
=\bwedge{2k'}\Sym^{k+k'-1}E\otimes\Delta^{(\tilde n',j)}\Sym^k E=0,
\end{multline*}
because the first tensor factor in the last line is 0, as $k+k'<2k'$. We plug this into the previous computation to conclude 
\[\what{\mathcal{M}}_j\pi(A_i\otimes B_i)=\pi(A_i\otimes 0)=0,\] 
and so $S_i$ is a submodule of $M_{j+1}$. In particular, by induction, $S_i$ is a submodule of $M_i$.
\end{proof}

\begin{lemma}\label{lem:phi-surj-in-thm1.1}
    As a linear map, the restriction $\mathcal{M}_i:S_i\to\tilde\pi(L_i)$ is surjective.
\end{lemma}
\begin{proof}
    As in the proof above, we begin with observing that 
    \begin{align}\label{eq:Mi-pi-commute}
    \mi(S_i)=\hmi\pi(A_i\otimes B_i)=\pi(A_i\otimes \hmi(B_i)).
    \end{align}

    We now compute explicitly the image under $\hmi$ of a basis element $s_\lambda\otimes s_\mu$ of~$B_i$, where $\lambda\in L(n,k)$ and $\mu\in L(m,k)$, that is: 
    \[
        \hmi(s_\lambda\otimes s_\mu) 
        =(\wr_{2k}^{-1}\otimes 1)(\extproduct_{(2^k)}\otimes\psi_{\tilde ni})(\wr_{k}\otimes1)^{\otimes2}(\symcoproduct_{k\tilde n}\otimes\symcoproduct_{ki}) s_\lambda\otimes s_\mu\,.
    \]
    We can rewrite the right-hand side as
    \[
    (\wr_{2k}^{-1}\otimes \psi_{\tilde n i})(\extproduct_{(2^k)}\otimes1)(\wr_{k}\otimes1)^{\otimes2}(\symcoproduct_{k\tilde n}\otimes\symcoproduct_{ki}) s_\lambda\otimes s_\mu\,.
    \]
    By Remark~\ref{rem:sym-coprod-schur} and by~\eqref{eq:wr as poly} we compute
    \begin{align*}   
    (\wr_{k}\otimes1)^{\otimes2}&(\symcoproduct_{k\tilde n}\otimes\symcoproduct_{ki}) s_\lambda\otimes s_\mu
    \\
    &= (\wr_{k}\otimes1)^{\otimes2}\sum_{\substack{\nu\in L(k,k)\\\nu\subseteq\lambda\subseteq\nu\sqcup(k^{\tilde n})}}\sum_{\substack{\theta\in L(k,k)\\\theta\subseteq\mu\subseteq\theta\sqcup(k^i)}} s_\nu\otimes s_{\lambda/\nu}\otimes s_\theta\otimes s_{\mu/\theta}
    \\
    &= \sum_{\substack{\nu\in L(k,k)\\\nu\subseteq\lambda\subseteq\nu\sqcup(k^{\tilde n})}}\sum_{\substack{\theta\in L(k,k)\\\theta\subseteq\mu\subseteq\theta\sqcup(k^i)}} a_{\nu+\rho_k}\otimes s_{\lambda/\nu}\otimes a_{\theta+\rho_k}\otimes s_{\mu/\theta}\,.
    \end{align*}
    Apply $\extproduct_{(2^k)}\otimes 1$ to the resulting expression using Remark~\ref{rmk:ext-product} to obtain
    \[
    \sum_{\substack{\nu,\theta \in L(k,k)\\\nu\subseteq\lambda\subseteq\nu\sqcup(k^{\tilde n})\\\theta\subseteq\mu\subseteq\theta\sqcup(k^i)}} \sgn_{kk}(\nu,\theta) a_{(\nu+\rho_k)\sqcup( \theta+\rho_k)}\otimes s_{\lambda/\nu}\otimes s_{\mu/\theta} 
    = a_{\rho_{2k}}\otimes\varphi(s_\lambda\otimes s_\mu),
    \]
    where we have used Corollary~\ref{cor:flip-alternating-poly} with $\varphi$ as defined in~\eqref{eq:def-phi}.
    Finish by applying $\wr_{2k}^{-1}\otimes\psi_{\tilde n i}$ to get
    \begin{align}\label{eq:Mi-on-basis-Bi}
    \hmi(s_\lambda\otimes s_\mu) 
    =
    1\otimes\psi_{\tilde ni}\,\varphi(s_\lambda\otimes s_\mu)\,.
    \end{align}

    Now, combining \eqref{eq:Mi-pi-commute}, \eqref{eq:Mi-on-basis-Bi}, and the surjectivity of $\varphi$ from Corollary~\ref{cor:phi-surj}, we conclude:
    \begin{align*}
        \mi(S_i)&=\pi(A_i\otimes \hmi(B_i))=\pi(A_i\otimes \psi_{\tilde n i}\varphi(B_i))\\&=\pi(A_i\otimes \psi_{\tilde n i}(\Sym_{(\tilde n,i)}\Sym^k E))
        =\tilde \pi(A_i\otimes \Delta^{(\tilde n,i)}\Sym^kE)=\tilde\pi(L_i).\hfill\qedhere
    \end{align*}
\end{proof}

\subsection{Dimension counting and exactness}\label{sec:dimension}
\begin{proposition}\label{prop:surj-Li}
    The $\SL_2(\F)$-homomorphisms $\tilde \pi^{-1}_{(\tilde n,i),2k,k,d-k}\circ\mathcal{M}_i$ realise the surjections
    \(
    M_i\twoheadrightarrow L_i.
    \)
\end{proposition}
\begin{proof}
    In the previous section we showed that
    \begin{equation}\label{eq:containment}
    L_i\cong\tilde\pi(L_i)=\mathcal{M}_i(S_i)\leq \im\mathcal{M}_i
    \end{equation}
    as $\F$-vector spaces.
    Hence, by definition of $M_{i+1}$ as the kernel of $\mathcal{M}_i$ and by rank-nullity,
    \[
    \dim L_i\leq \dim(\im\mathcal{M}_i)=\dim M_i-\dim M_{i+1}
    \]
    for $i=0,1,\ldots,m-1$.
    Summing these inequalities, we get
    \[
    \sum_{i=0}^{m-1}\dim L_i\leq \dim M_0-\dim M_m\le\dim M_0\,.
    \]
    Recall that $M_0=\Delta^{(n,m)}\Sym^d \E$ by definition, and by equation~\eqref{eq:main} with $q=1$, $\sum\dim L_i=\dim\Delta^{(n,m)}\Sym^d \E$, so every inequality above is in fact an equality, in particular $\dim L_i=\dim (\im\mathcal{M}_i)$ and also $\dim M_m=0$. Together with the relation~\eqref{eq:containment} this gives $\im\mathcal{M}_i~\cong~L_i$, now as $\SL_2(\F)$-modules because $\tilde \pi$ and $\mathcal{M}_i$ are $\SL_2(\F)$-homomorphisms.
\end{proof}

\begin{proof}[Proof of Theorem~\ref{thm:main}]
    By Proposition~\ref{prop:surj-Li}, for $i=0,1,\ldots,m-1$ there is a short exact sequence
    \[
    0\to M_{i+1}\to M_i\xrightarrow{\tilde \pi^{-1}\circ\mathcal{M}_i} L_i \to 0,
    \]
    which proves the claimed filtration of $\Delta^{(n,m)}\Sym^d \E$.
\end{proof}

As a corollary of our main result, we get the following structure theorem.
\begin{Theorem}
	There is a long exact sequence of $\SL_2(\F)$-modules 
	\[
	\cdots
	\to \Delta^{(2,2)}\Sym^{d}\E 
	\to \Delta^{(2,2)}\Sym^{d-1}\E
	\to \cdots \to 
	\Delta^{(2,2)}\Sym^{2}\E \to \F.
	\]
\end{Theorem}
\begin{proof}
	Theorem~\ref{thm:main} with $n=m=2$ gives a short exact sequence
	\[
	0\to\Sym_4\Sym^{d-1}\E \to \Delta^{(2,2)}\Sym^d\E \xrightarrow{\overline{\mathcal{M}}_1^{(d)}}
	\Sym_4\Sym^{d-2}\E \to 0\,
	\]
	for every $d\ge 2$. In particular \[\ker\overline{\mathcal{M}}_1^{(d)}=\Sym_4\Sym^{d-1} E=\im\overline{\mathcal{M}}_1^{(d+1)},\] so the homomorphisms $\overline{\mathcal{M}}_1^{(d)}$ for $d\ge2$ form the long exact sequence above.
\end{proof}

\section*{Acknowledgements}
The authors thank Eoghan McDowell for helpful discussions on his construction of Weyl modules.
    
    \newcommand{\etalchar}[1]{$^{#1}$}


\begin{thebibliography}{GMSW26b}

\bibitem[AFP{\etalchar{+}}19]{AFPRW}
M.~Aprodu, G.~Farkas, {\c{S}}.~Papadima, C.~Raicu, and J.~Weyman.
\newblock Koszul modules and {G}reen’s conjecture.
\newblock {\em Inventiones mathematicae}, 218(3):657--720, 2019.

\bibitem[ALRS13]{ALRS}
F.~Ardila, E.~Le{\'o}n, M.~Rosas, and M.~Skandera.
\newblock Tres lecciones en combinatoria algebraica. {II}. {L}as funciones
  sim\'etricas y la teor\'ia de representaciones.
\newblock 2013.
\newblock arXiv:1301.3988.

\bibitem[BB05]{BjornerBrenti}
A.~Bjorner and F.~Brenti.
\newblock {\em Combinatorics of {C}oxeter groups}.
\newblock Springer, 2005.

\bibitem[BFK99]{BFK}
J.~Bernstein, I.~Frenkel, and M.~Khovanov.
\newblock A categorification of the {T}emperley--{L}ieb algebra and {S}chur
  quotients of ${U}(\mathfrak{sl}_2)$ via projective and {Z}uckerman functors.
\newblock {\em Selecta Mathematica}, 5:199--241, 1999.

\bibitem[CF94]{CraneFrenkel}
L.~Crane and I.~B. Frenkel.
\newblock Four-dimensional topological quantum field theory, {H}opf categories,
  and the canonical bases.
\newblock {\em Journal of Mathematical Physics}, 35(10):5136--5154, 1994.

\bibitem[FGGK26]{AffineJT}
I.~Fischer, M.~Gangl, {\'A}.~Gutiérrez, and N.~Kumari.
\newblock Affine {J}acobi--{T}rudi formulas.
\newblock In preparation, 2026+.

\bibitem[FKS07]{FKS}
I.~Frenkel, M.~Khovanov, and C.~Stroppel.
\newblock A categorification of finite-dimensional irreducible representations
  of quantum $\mathrm{sl}_2$ and their tensor products.
\newblock {\em Selecta Mathematica}, 12(3):379, 2007.

\bibitem[Ful97]{FultonYT}
William Fulton.
\newblock {\em Young tableaux}, volume~35 of {\em London {M}athematical
  {S}ociety student texts}.
\newblock CUP, 1997.

\bibitem[GK25]{GutKrat}
{\'A}.~Gutiérrez and C.~Krattenthaler.
\newblock Schur log-concavity and the quantum {P}ascal triangle.
\newblock arXiv:2509.22648, 2025.

\bibitem[GMSW25]{hooks}
Á. Gutiérrez, {\'A}.~Martínez, M.~Szwej, and M.~Wildon.
\newblock Modular isomorphisms of $\mathrm{SL}_2(\mathbb{F})$-plethysms for
  {W}eyl modules labelled by hook partitions.
\newblock arXiv:2509.01490, 2025.

\bibitem[GMSW26a]{GSMW-Pfaff}
{\'A}.~Gutiérrez, {\'A}.~Martínez, M.~Szwej, and M.~Wildon.
\newblock A new bijective proof of the $q$-{P}faff--{S}aalsch\"utz identity
  with applications to quantum groups.
\newblock {\em European Journal of Combinatorics}, 133:104321, 2026.

\bibitem[GMSW26b]{GMSW-FPSAC}
{\'A}.~Gutiérrez, {\'A}.~Martínez, M.~Szwej, and M.~Wildon.
\newblock Plethystic lifts of $q$-binomial identities.
\newblock {\em To appear in the Proceedings of FPSAC, S\'eminaire Lotharingien
  Combinatoire}, 2026.

\bibitem[Gri23]{Grinberg}
D.~Grinberg.
\newblock Comments on arxiv:2105.00538v3.
\newblock Available at
  \url{https://www.cip.ifi.lmu.de/~grinberg/algebra/sln-equiv.pdf}, 2023.

\bibitem[IOT25]{IOT}
C.~Ikenmeyer, H.~Omar, and D.~Tsintsilidas.
\newblock Field-independent {K}ronecker-plethysm isomorphisms, 2025.
\newblock arXiv:2509.10069.
\bibitem[KL09]{KhovanovLauda}
M. Khovanov and A. D. Lauda.
\newblock A diagrammatic approach to categorification of quantum groups {I}
\newblock {\em Representation Theory of the American Mathematical Society}, 13(14),
2009.

\bibitem[Lau10]{Lauda}
A.~D. Lauda.
\newblock A categorification of quantum {${\rm sl}(2)$}.
\newblock {\em Advances in Mathematics}, 225(6), 2010.

\bibitem[LR11]{LoehrRemmel}
Nicholas~A. Loehr and Jeffrey~B. Remmel.
\newblock A computational and combinatorial expos\'e of plethystic calculus.
\newblock {\em J. Algebraic Combin.}, 33(2):163--198, 2011.

\bibitem[Mac95]{Macdonald}
I.~G. Macdonald.
\newblock {\em Symmetric functions and {H}all polynomials}.
\newblock Oxford Mathematical Monographs. The Clarendon Press Oxford University
  Press, New York, second edition, 1995.
\newblock With contributions by A. Zelevinsky, Oxford Science Publications.

\bibitem[McD25]{McDowell25}
E.~McDowell.
\newblock An explicit construction of the {W}eyl module as a quotient of
  symmetric tensors by dual {G}arnir relations.
\newblock arXiv:2508.14788, 2025.

\bibitem[McDW22]{McDW}
E.~McDowell and M.~Wildon.
\newblock Modular plethystic isomorphisms for two-dimensional linear groups.
\newblock {\em Journal of Algebra}, 602, 2022.

\bibitem[MW25]{MW}
{\'A}.~Martínez and M.~Wildon.
\newblock A new modular plethystic {${\rm SL}_2(\mathbb{F})$}-isomorphism {${\rm
  Sym}^{N - 1}E\otimes \wedge ^{N + 1}{\rm Sym}^{d + 1}E\cong \Delta^{( 2, 1^{N
  - 1} )}{\rm Sym}^dE$}.
\newblock {\em Journal of Algebra}, 682, 2025.

\bibitem[Rou08]{Rouquier}
R. Rouquier.
\newblock 2-{K}ac-{M}oody algebras.
\newblock
arxiv:0812.5023, 2008.

\bibitem[Sta99]{StanleyEC2}
R.~P. Stanley.
\newblock {\em Enumerative combinatorics Vol. 2}.
\newblock Cambridge Studies in Adv. Math., 1999.

\bibitem[Str05]{Stroppel}
C.~Stroppel.
\newblock Categorification of the {T}emperley--{L}ieb algebra, tangles and
  cobordisms via projective functors.
\newblock {\em Duke Mathematical Journal}, 126(3):547--596, 2005.

\bibitem[Wey03]{Wey03}
J.~Weyman.
\newblock {\em Cohomology of Vector Bundles and Syzygies}.
\newblock Cambridge Tracts in Mathematics. Cambridge University Press, 2003.

\end{thebibliography}
\end{document}